\documentclass[reqno]{amsart}
\usepackage{amscd, amsmath, amsthm, amssymb, amsmath,mathrsfs}
\usepackage[all]{xy}
\newtheorem{thm}{Theorem}[section]
\newtheorem{lem}[thm]{Lemma}
\newtheorem{prop}[thm]{Proposition}

\newtheorem{cor}[thm]{Corollary}

\newtheorem*{thmA1}{Theorem A.1}
\newtheorem*{thmA3}{Theorem A.3}
\newtheorem*{thmA5}{Theorem A.5}

\theoremstyle{definition}     
\newtheorem{definition}[thm]{Definition}

\theoremstyle{remark}
\newtheorem{rem}[thm]{Remark}
\newtheorem*{remA2}{Remark {\rm{A.2}}}
\newtheorem*{remA4}{Remark {\rm{A.4}}}
\newtheorem*{remA6}{Remark {\rm{A.6}}}

\numberwithin{equation}{section}

\newcommand{\bF}{\mathbb F}
\newcommand{\bP}{\mathbb P}
\newcommand{\bZ}{\mathbb Z}

\newcommand{\ctext}[1]{\textcircled{\scriptsize{#1}}}

\newcommand{\mapright}[1]{\smash{\mathop{\longrightarrow}\limits^{#1}}}

\newcommand{\Xt}{\tilde{X}}
\newcommand{\sh}[1]{\mathcal{O}_{#1}}
\newcommand{\Kx}{K_X}
\newcommand{\Kxt}{K_{\tilde{X}}}
\newcommand{\Xtp}{\tilde{X}^{\prime}}
\newcommand{\Kxtp}{K_{\tilde{X}^{\prime}}}
\newcommand{\ch}{{\rm{char}}}
\newcommand{\Spec}{{\rm{Spec}}}
\newcommand{\etap}{\eta ^{\prime}}
\newcommand{\kb}{\bar{k}}
\newcommand{\mA}{{\mathcal{A}}}
\newcommand{\mB}{{\mathcal{B}}}
\newcommand{\mC}{{\mathcal{C}}}
\newcommand{\mD}{{\mathcal{D}}}
\newcommand{\mAp}{{\mathcal{A}^{\prime}}}
\newcommand{\mBp}{{\mathcal{B}^{\prime}}}
\newcommand{\mCp}{{\mathcal{C}^{\prime}}}
\newcommand{\Cp}{C^{\prime}}
\newcommand{\Dp}{D^{\prime}}
\newcommand{\Ep}{E^{\prime}}
\newcommand{\Fp}{F^{\prime}}
\newcommand{\Gp}{G^{\prime}}
\newcommand{\mF}{{\mathcal{F}}}

\newcommand{\Xkb}{X\times_k \kb}
\newcommand{\Vt}{\tilde{V}}
\newcommand{\Vtp}{\tilde{V}^{\prime}}
\newcommand{\Vp}{V^{\prime}}
\newcommand{\Kvp}{K_{V^{\prime}}}
\newcommand{\Proj}{{\rm{Proj}}}
\newcommand{\Sp}{S^{\prime}}
\newcommand{\Tr}{{\rm{Tr}}}
\newcommand{\Gal}{{\rm{Gal}}}
\newcommand{\Pic}{{\rm{Pic}}}

\newcommand{\Imr}{{\rm{Im}}}
\newcommand{\mL}{{\mathcal{L}}}

\begin{document}

\title{Unirationality of RDP Del Pezzo surfaces of degree 2}

\author{Ryota Tamanoi}
\email{ryotatamanoi@gmail.com}

\address{Graduate School of Mathematical Sciences, The University of Tokyo, 3-8-1 Komaba, Meguro-ku, Tokyo 153-8914, Japan}

\begin{abstract}
We study unirationality of a Del Pezzo surface of degree two over a given (non algebraically closed) field, under the assumption that it admits at least one rational double point over an algebraic closure of the base field. As corollaries of our main results, we find that over a finite field, it is unirational if the cardinality of the field is greater than or equal to nine and we also find that over an infinite field, which is not necessarily perfect, it is unirational if and only if the rational points are Zariski dense over the field. 
\end{abstract}
\maketitle

\section{Introduction}
	Let $X$ be a projective, normal and geometrically integral surface over a field $k$ and $\kb$ be an algebraic closure of $k$. $X$ is a smooth Del Pezzo surface if $X$ is smooth and the anti-canonical divisor $-\Kx$ is ample. As a generalization of this, if $X$ is singular but has only rational double points and $-\Kx$ is ample, $X$ is called a RDP Del Pezzo surface. It is well known that a smooth Del Pezzo surface over an algebraically closed field $\kb$ is $\kb$-rational, that is, birational to $\bP_{\kb}^2$. A RDP Del Pezzo surface over $\kb$ is also $\kb$-rational. However, $X$ over a field $k$ is not always $k$-rational (even if $X$ has a smooth $k$-point), so we consider the weaker condition unirationality. A projective variety $V$ over $k$ is $k$-unirational if there exists a dominant map $\bP_k^n \dashrightarrow V$. In this article, we study about the unirationality of RDP Del Pezzo surfaces.

	The self intersection number of $-\Kx$ is called the degree of a smooth or RDP Del Pezzo surface $X$, denoted by $d$. Let $X_1$ be a smooth Del Pezzo surface of degree $d$. We also suppose that $X_1$ has a rational point since varieties with no rational point are clearly not unirational. For $d \geq 5$, Manin proved that $X_1$ is $k$-rational over arbitrary fields {\cite[Theorem 29.4]{Manin}}. Manin also proved $X_1$ is $k$-unirational if $d = 3,4$ and $X_1$ has a rational point which does not lie on any exceptional curves. The unirationality of $X_1$ for $d = 3$ and $4$ was finally proved in {\cite[Theorem 1.1]{Kollar}} and {\cite[Proposition 5.19]{Pieropan}}, respectively. Let $X_2$ be a RDP Del Pezzo surface of degree $d$ with a rational point over a perfect field $k$. {\cite[Section 5]{Segre2}}, {\cite[Proposition 1]{Coray}} and {\cite[Theorem A]{CT}} proved that $X_2$ is $k$-unirational if $d = 3$. {\cite{CT}} also proved the $k$-unirationality of $X_2$ for $d \geq 4$. These propositions for RDP Del Pezzo surfaces are detailed in Section 2.2. 
\begin{thm}{\label{MainThm1}}
	Let $X$ be a RDP Del Pezzo surface of degree $2$ over a perfect field $k$ and $\Xt$ be a minimal resolution of $X$. Suppose that the singularities of $X$ are neither
\begin{itemize}
	\item[$\ctext{1}$] $A_1$ type (that is, the singularity of $X$ over $\kb$ is a unique $A_1$ type singularity),
	\item[$\ctext{2}$] $A_2$ type and the two $(-2)$curves on $\Xt$ are conjugate under $\Gal(\kb / k)$,
	\item[$\ctext{3}$] $4A_1$ type and the four singularities are conjugate under $\Gal(\kb / k)$,
\end{itemize}
where $\Gal ( \kb / k)$ is the absolute Galois group of $k$. Then, $\Xt$ is not minimal over $k$. In particular, $X$ is $k$-unirational if $\Xt$ has a $k$-point.
\end{thm}
	The type of singularities of $X$ is one of the $40$ types in Proposition {\ref{Prop:singularity type}}. Theorem {\ref{MainThm1}} states all but $3$ cases of these are non-minimal.
  
	For a smooth Del Pezzo surface $X$ of degree $2$ over a field $k$, {\cite{STVA}} gave a necessary and sufficient condition of unirationality: $X$ is $k$-unirational if there exists a non-constant morphism $\bP_k^1 \rightarrow X$ ({\cite[Theorem 3.2]{STVA}}). As the corollary of this, {\cite[Corollary 3.3]{STVA}} states that $X$ is $k$-unirational if $X$ has a rational point which is not a generalized Eckardt point (cf. Definition {\ref{Def:geneEckpt}}) or on the ramification divisor of anti-canonical morphism of $X$. The following second main theorem is the RDP Del Pezzo surface version of this corollary.
\begin{thm}{\label{MainThm2}}
	Let $X$ be a RDP Del Pezzo surface of degree $2$ over a field $k$ and $\kappa : X \rightarrow \bP_k^2$ be the anti-canonical morphism. Suppose either
\begin{itemize}
	\item[(1)] $\kappa$ is purely inseparable and $k$ is perfect or
	\item[(2)] $\kappa$ is separable and $X$ has a $k$-point which is neither a generalized Eckardt point nor on the ramification divisor of $\kappa$.
\end{itemize}
Then, $X$ is $k$-unirational.
\end{thm}
In particular, for an infinite field $k$, $X$ is $k$-unirational if and only if the set of rational points $X(k)$ is Zariski dense in $X$.

	Moreover, {\cite{STVA}} and {\cite{FL}} showed that if $k$ is a finite field, then smooth Del Pezzo surfaces of degree $2$ are $k$-unirational({\cite[Theorem 1.1]{FL}}). The following third main theorem is the RDP Del Pezzo surface version of this theorem.
\begin{thm}{\label{MainThm3}}
	Let $X$ be a RDP Del Pezzo surface of degree $2$ over a finite field $k$ with $q$ elements. Suppose either:
\begin{itemize}
	\item[(1)] $X$ is not $\ctext{1}$, $\ctext{2}$ or $\ctext{3}$ of Theorem {\ref{MainThm1}}; or
	\item[(2)] $X$ is either $\ctext{1}$, $\ctext{2}$ or $\ctext{3}$ of Theorem {\ref{MainThm1}} and $q$ is at least $n$, where
$$
n = 
\begin{cases}
	9 & $if $\ctext{1}, \\
	8 & $if $\ctext{2}, \\
	4 & $if $\ctext{3}.
\end{cases}
$$
\end{itemize}
Then, $X$ is $k$-unirational.
\end{thm}
	I do not know whether $X$ is $k$-unirational when $X$ is of type $\ctext{1}$, $\ctext{2}$, $\ctext{3}$ and $q < 9, 8, 4$, respectively.

	This article is structured as following: Section 2 describes the properties of RDP Del Pezzo surfaces of any degree. In Section 3, we show Theorem {\ref{MainThm1}}. In Section 4, we show Theorem {\ref{MainThm2}}. This section is independent from Section 3. In Section 5, we show Theorem {\ref{MainThm3}}. Many of propositions and proofs of Section 4 and 5 are similar to {\cite{STVA}}.
\section{Geometry of RDP Del Pezzo surfaces}
	Let $X$ be a RDP Del Pezzo surface. We call the self-intersection number $d = \Kx \cdot \Kx$ the degree of $X$. Such surfaces over an algebraically closed field $\kb$ are studied by Demazure in \cite{Demazure} and Hidaka and Watanabe in \cite{HW}. Some important results of \cite{Demazure} and \cite{HW} are cited in Section 2.1. In Section 2.2, results of {\cite{CT}} about  the unirationality of RDP Del Pezzo surfaces of degree at least $3$ are cited.
\subsection{RDP Del Pezzo surfaces and weak Del Pezzo surfaces}
	Let $\Sigma = \{ x_1,\cdots ,x_r\} $ be a finite set of points on $\bP _{\kb} ^2$ ($1\leq r \leq 8$), infinitely near points allowed. Denote by $\Sigma _j$ the subset $\{ x_1,\cdots ,x_j\}$ ($1\leq j\leq r$) and let $V(\Sigma _j)\rightarrow \bP_{\kb} ^2$ be the blowing up of $\bP_{\kb} ^2$ with center $\Sigma _j$. Then there exists a sequence of blowing ups
	$$V(\Sigma ) = V(\Sigma _r)\longrightarrow V(\Sigma _{r-1})\longrightarrow \cdots \longrightarrow V(\Sigma _1) \longrightarrow \bP_{\kb} ^2$$
Let $E_j$ be the exceptional divisor which is inverse image of $x_j$ by the blowing up $V(\Sigma _j)\rightarrow \bP_{\kb} ^2$. 
\begin{definition}[\cite{Demazure}]
The points of $\Sigma \subset \bP _{\kb}^2$ are in {\em{general position}} (resp. {\em{almost general position}}) if
\begin{itemize}
	\item[(i)] no three (resp. four) of them are on a line.
	\item[(ii)] no six (resp. seven) of them are on a conic.
	\item[(iii)] all the points are distinct (resp. for all $j$ ($1\leq j \leq r-1$), the point $P_{j+1}\in V(\Sigma _j)$ does not lie on any proper transform $\hat{E_j}$ of $E_j$ such that $\hat{E_j}^2 = -2$).
	\item[(iv)] when $r = 8$, there exists no singular cubic which passes through all the points of $\Sigma$ and has one of them as the singular point (no corresponding condition for almost general position).
\end{itemize}
\end{definition}
	It is well known that a projective surface $X$ is a smooth Del Pezzo surface over $\kb$ if and only if $X$ is isomorphic to either $\bP_{\kb} ^2$, $\bP_{\kb} ^1 \times \bP_{\kb} ^1$ or $V(\Sigma)$ where the points of $\Sigma$ are in general position. A similar statement for almost general position holds:
\begin{prop}[{\cite[III Th\'eor\`eme 1]{Demazure}}]\label{Prop:Demazure}
Let $V = V(\Sigma )$ as above. Then the following conditions are equivalent:
\begin{itemize}
	\item[(a)] The points of $\Sigma$ are in almost general position.
	\item[(b)] The anti-canonical system $|{-K_{V}}|$ of $V$ has no fixed components.
	\item[(c)] $H^1(V,\sh{V}(nK_{V}))=0$ for all $n \in \bZ$.
	\item[(c$^{\prime}$)] There exists a series of integers $\{ n_i\}$, tending to $-\infty$ such that $H^1 (V,\sh{V}(n_i K_{V}))=0$
	\item[(d)] $D \cdot K_{V} \leq 0$ for every effective divisor $D$ on $V$.
	\item[(d$^{\prime}$)] For an irreducible curve $D$ on $V$, either $D \cdot K_{V} < 0$ or $D\cdot K_{V}=0$ and $D^2=-2$.
\end{itemize}
	Moreover, if $\ch(\kb) = 0$, the following condition is also equivalent:
\begin{itemize}
	\item[(b$^{\prime}$)] The anti-canonical system $|{-K_{V}}|$ of $V$ contains a non-singular elliptic curve.
\end{itemize}  
\end{prop}
	Therefore, $V$ is a weak Del Pezzo surface if $V$ satisfies one of (thus all) the above statements, where a surface $V$ is called a weak Del Pezzo surface if $-K_{V}$ is nef. Demazure also proved the following:
\begin{prop}[{\cite{Demazure}}, {\cite[Proposition 0.6]{CT}}]{\label{collapsing weakDelPezzo}}
	Let $\Vp$ be a weak Del Pezzo surface of degree $d$ over an algebraically closed field $\kb$. Let $f_0 : \Vp \rightarrow V_0$ be the rational map defined by the complete linear system $|-i\Kvp|$, where
$$
i = 
\begin{cases}
	1 & $if $d \geq 3, \\
	2 & $if $d = 2, \\
	3 & $if $d = 1.
\end{cases}
$$
Then $f_0$ is a morphism which contracts all the $(-2)$curves of $\Vp$ and is an isomorphism everywhere else. Its image $V_0$ is a RDP Del Pezzo surface. For $d \geq 3$, $V_0$ is anticanonically embedded as a surface of degree $d$ in $\bP^d$.
\end{prop}
{\cite[Theorem 4.4]{HW}} states that for $d = 2$ (resp.$1$), $V_0$ is also embedded as a surface of degree $4$ (resp.$6$) in the weighted projective space $\bP(1,1,1,2)$ (resp.$\bP(1,1,2,3)$).

	Hidaka and Watanabe proved the converse of Proposition {\ref{Prop:Demazure}}:
\begin{prop}[{\cite[Theorem 3.4]{HW}}]{\label{Prop:HW}}
Let $X$ be a smooth or RDP Del Pezzo surface over an algebraically closed field $\kb$ and $f : \Xt \rightarrow X$ be a minimal resolution of $X$. Then
\begin{itemize}
	\item[(i)] $1 \leq d = \Kx \cdot \Kx \leq 9$.
	\item[(ii)] If $d = 9$, then $X \cong \bP_{\kb} ^2$.
	\item[(iii)] If $d = 8$, then either $(a) X \cong \bP_{\kb} ^1 \times \bP_{\kb} ^1$ or $(b) X \cong \bF _1$ or $(c) X$ is the cone over a quadric in $\bP_{\kb} ^2$. In this case, $\Xt \cong \bF _2$ and the resolution $f$ is given by contracting the minimal section of $\Xt$, where $\bF_n$ is a Hirzebruch surface.

	\item[(iv)] If $1 \leq d \leq 7$, then there exists a set of $r$ points on $\bP_{\kb} ^2$ such that the points of $\Sigma$ are in almost general position, $r = 9-d$ and $\Xt \cong V(\Sigma )$. In this case,	the resolution $f$ is the contraction of all curves on $\Xt$ with self-intersection number $-2$.
\end{itemize}
\end{prop}
	From the results above, the following corollary holds.
\begin{cor}{\label{Cor:RDPDP<->weakDP}}
Let $X$ and $\Xt$ be a RDP Del Pezzo surface and a weak Del Pezzo surface over a field $k$, respectively.  
\begin{itemize}
	\item[(1)] The minimal resolution of $X$ is a weak Del Pezzo surface. In particular, the singularities of $X$ come from $(-2)$curves of the weak Del Pezzo surface.

	\item[(2)] By collapsing $(-2)$curves of $\Xt$, $\Xt$ becomes a RDP Del Pezzo surface. In particular, $(-2)$curves of $\Xt$ come from singularities of the RDP Del Pezzo surfaces.
\end{itemize}
\end{cor}

	The following proposition is used in Section $4$.
\begin{prop}[{\cite[Proposition 4.2 (iii)]{HW}}]{\label{dim of anticano}}
	Let $X$ be a smooth or RDP Del Pezzo surface of degree $d$. Then,
$$
\dim H^0 (X, \sh{X}(-m\Kx)) = 
\begin{cases}
	d \cdot m(m + 1)/2 + 1 & (m \geq 0) \\
	0 & (m < 0)
\end{cases}
$$
and $\dim H^1 (X, \sh{X}(-m\Kx)) = 0$ for all $m \in \bZ$.
\end{prop}
\subsection{Unirationality of RDP Del Pezzo surfaces of degree at least 3}
	In this subsection, $k$ is a perfect field. The unirationality of RDP Del Pezzo surfaces of degree $d \geq 3$ over a perfect field is studied in {\cite{CT}} in detail.

\begin{prop}[{\cite[Theorem A]{CT}}]{\label{Prop:bir degree3}}
	Let $V$ be a RDP Del Pezzo surface of degree $3$ over a perfect field $k$ and $\delta$ be the number of singularities of $V \times_k \kb$. Then $\delta \leq 4$ and $V$ is birationally equivalent (over $k$) to:
\begin{itemize}
	\item[] $\bP^2_k$ if $\delta = 1$ or $4$;
	\item[] a smooth Del Pezzo surface of degree $4$ with a $k$-point if $\delta = 2$;
	\item[] a smooth Del Pezzo surface of degree $6$ if $\delta = 3$.
\end{itemize}
\end{prop}
It is also known that $V$ is $k$-rational if $V$ has a singular $k$-point ({\cite[Section 4]{Segre2}}). Thus it follows Proposition {\ref{Prop:bir degree3}} that RDP Del Pezzo surfaces of degree $3$ with a $k$-point are $k$-unirational.
\par
For $d = 4$, RDP Del Pezzo surfaces are classified as below:
\begin{prop}[{\cite{DuVal}}, {\cite[Proposition 5.6]{CT}}]{\label{Prop:classify degree4}}
	Let $V$ be a RDP Del Pezzo surface of degree $4$ over an algebraically closed field $\kb$ and $\Vt$ be a minimal resolution of $V$. Then $V$ satisfies one of the following:
\begin{itemize}
	\item[1.] $V$ has $A_1$ singularity;
	\item[2.] $V$ has $2A_1$ singularities and $\Vt$ has one $(-1)$curve which intersects both of two $(-2)$curves;
	\item[3.] $V$ has $2A_1$ singularities and $\Vt$ has no $(-1)$curve which intersects both of two $(-2)$curves;
	\item[4.] $V$ has $A_2$ singularity;
	\item[5.] $V$ has $3A_1$ singularities;
	\item[6.] $V$ has $A_1 + A_2$ singularities;
	\item[7.] $V$ has $A_3$ singularity and $\Vt$ has five $(-1)$curves; 
	\item[8.] $V$ has $A_3$ singularity and $\Vt$ has four $(-1)$curves; 
	\item[9.] $V$ has $4A_1$ singularities;
	\item[10.] $V$ has $2A_1 + A_2$ singularities;
	\item[11.] $V$ has $A_1 + A_3$ singularities;
	\item[12.] $V$ has $A_4$ singularity;
	\item[13.] $V$ has $D_4$ singularity;
	\item[14.] $V$ has $2A_1 + A_3$ singularities or;
	\item[15.] $V$ has $D_5$ singularity;
\end{itemize}
\end{prop}
	{\cite[Proposition 6.1]{CT}} has the configuration of $(-1)$curves and $(-2)$curves of $\Vt$ by the type of singularities of $V$.

	There are two different equivalent classes if $2A_1$ type or $A_3$ type. One case of type $2A_1$ is called an Iskovskih surface.
\begin{definition}{\label{Def:Iskovski surface}}
	Let $V$ be a RDP Del Pezzo surface of degree $4$ over $k$ which is in Case $3$ of Proposition {\ref{Prop:classify degree4}}. If the two singularities over $\kb$ are conjugate, then $V$ is called an {\em{Iskovskih surface}}.
\end{definition}

	If $V$ is not an Iskovskih surface, it follows from the following proposition that $V$ with a smooth $k$-point is $k$-rational.

\begin{prop}[{\cite[Lemma 7.4]{CT}}]{\label{Prop:bir degree4}}
	Let $V$ be a RDP Del Pezzo surface of degree $4$ over a perfect field $k$ and let $\Vt$ be a minimal resolution of $V$. If $V$ is not an Iskovskih surface, $\Vt$ is not minimal over $k$. In fact, there exists a birational morphism $\Vt \rightarrow \Vtp$ such that:
\begin{itemize}
	\item[] in Cases $2, 6, 7, 11, 12$ and $15$, of Proposition {\ref{Prop:classify degree4}}, $\Vtp$ is isomorphic to $\bP^2_k$;
	\item[] in Cases $1, 4, 5, 9, 10$ and $14$, $\Vtp$ is a form of $\bP^1_{\kb} \times \bP^1_{\kb}$;
	\item[] in Cases $3, 8$ and $13$, $\Vtp$ is a weak Del Pezzo surface of degree $8$ with a $(-2)$curve, which is birational to a form of $\bP^1_{\kb} \times \bP^1_{\kb}$.
\end{itemize}
In particular, $V$ is $k$-rational if $\Vt$ has a $k$-point.
\end{prop}
Even if $V$ is an Iskovskih surface, $V$ is $k$-unirational when $V$ has a $k$-point since the surface given by blowing up at the $k$-point is a RDP Del Pezzo surface of degree $3$. Coray and Tsfasman also proved that non-minimal Iskovskih surface cannot be $k$-rational ({\cite[Proposition 7.7]{CT}}).

	For $d \geq 5$, the unirationality was proved as below.
\begin{prop}[{\cite[Corollary 9.4]{CT}}]{\label{Prop:bir degree>5}}
	Let $\Vt$ be a weak Del Pezzo surface of degree $d \geq 5$ over a perfect field $k$.
\begin{itemize}
	\item[(a)] If $d = 5$ or $7$, then $\Vt$ is $k$-rational.
	\item[(b)] If $\Vt$ has a $k$-point, then $\Vt$ is $k$-rational. 
\end{itemize}
\end{prop}
\section{Minimality of RDP Del Pezzo surfaces of degree 2}
	In this section, we show Theorem {\ref{MainThm1}}. The singularity of a RDP Del Pezzo surface of degree $2$ is classified by the following proposition:
\begin{prop}[{\cite{DuVal}}, {\cite[Section 8.7]{Dolgachev}}]{\label{Prop:singularity type}}
	Let $X$ be a RDP Del Pezzo surface of degree $2$ over an algebraically closed field and let $\delta$ be the number of singularities of $X$. Then $\delta \leq 7$ and the type is either
\begin{itemize}
	\item[(1)] $A_1, A_2, A_3, A_4, A_5, A_6, A_7, D_4, D_5, D_6, E_6$ or $E_7$ if $\delta = 1$;
	\item[(2)] $2A_1, A_1 + A_2, A_1 + A_3, A_1 + A_4, A_1 + A_5, A_1 + D_4, A_1 + D_5, A_1 + D_6,  2A_2, A_2 + A_3, A_2 + A_4, A_2 + A_5$ or $2A_3$ if $\delta = 2$;
	\item[(3)] $3A_1, 2A_1 + A_2, 2A_1 + A_3, 2A_1 + D_4, A_1 + 2A_2, A_1 + A_2 + A_3, A_1 + 2A_3$ or $3A_2$ if $\delta = 3$;
	\item[(4)] $4A_1, 3A_1 + A_2, 3A_1 + A_3$ or $3A_1 + D_4$ if $\delta = 4$;
	\item[(5)] $5A_1$ if $\delta = 5$;
	\item[(6)] $6A_1$ if $\delta = 6$;
	\item[(7)] $7A_1$ if $\delta = 7$ (this case occurs only if characteristic is $2$).	
\end{itemize}
\end{prop}
Theorem {\ref{MainThm1}} states that all but type $A_1, A_2$ and $4A_1$ are non-minimal.

	In Section 3.1, we prepare lemmas about $(-1)$curves on $\Xt$. The lemmas are also used in Section 4, so we do not assume $k$ is perfect in this subsection. In Section 3.2, 3.3, 3.4, 3.5 and 3.6, we show the case of $\delta = 2, 3, 4, \geq 5$ and $1$, respectively.
\subsection{Pre($-$1)curve}
	From now on, a {\em{($-$1)curve}} means a ``$(-1)$curve over $\kb$'', that is, a divisor $E$ on $\Xt \times_k \kb$ such that $E^2 = -1$ and $E \cong \bP_{\kb}^1$. If there exists a divisor $\Ep$ on $\Xt$ such that $\Ep \times_k \kb = E$, then $E$ is called a {\em{($-$1)curve defined over $k$}} and $E$ is identified with $\Ep$. Similarly, a {\em{($-$2)curve}} means a ``$(-2)$curve over $\kb$''.

	Let $X$ be a RDP Del Pezzo surface of degree $d$ over a field $k$ and let $\Xt$ be a minimal resolution of $X$. By Corollary {\ref{Cor:RDPDP<->weakDP}}, $\Xt$ is a weak Del Pezzo surface. Since results in this subsection are used in Section 4, we do not assume that $k$ is perfect or $d$ is $2$.
\begin{lem}{\label{Lem:exccurves}}
	For an irreducible divisor $C$ on $\Xt \times_k \kb$, the followings hold:
\begin{itemize}
	\item[(1)] If $C^2 = -1$, then $p_a(C)=0$. This means $C$ is a $(-1)$curve.
	\item[(2)] If $C^2 = -2$, then $p_a(C)=0$. This means $C$ is a $(-2)$curve.
	\item[(3)] $C\cdot (-\Kxt ) = 1$ if and only if either $C$ is a $(-1)$curve or $d=1$ and $C \in |{-\Kxt}|$.
	\item[(4)] $C\cdot (-\Kxt ) = 0$ if and only if $C$ is a $(-2)$curve.
\end{itemize}
\end{lem}
\begin{proof}
Since $C\cdot \Kxt + C^2 =2p_a(C) -2$ and $C \cdot \Kxt \leq 0$, (1) and (2) hold. (3) is {\cite[III Lemme 9]{Demazure}}. (4) follows from (2) and Proposition {\ref{Prop:Demazure}}.
\end{proof}

\begin{definition}
	Let $D$ be an effective divisor on $\Xt \times_k \kb$. $D$ is called a {\em{pre$(-1)$curve}} if $D^2 = D \cdot \Kxt = -1$.
\end{definition}
	By Lemma {\ref{Lem:exccurves}}, $(-1)$curves are pre$(-1)$curves. If $\Xt$ is a smooth Del Pezzo surface, every pre$(-1)$curve is a $(-1)$curve. We show some important properties of pre$(-1)$curves.

\begin{lem}\label{Lem3.2}
	Let $D$ be a pre$(-1)$curve on $\Xt$. $D$ contains a unique prime divisor $E$ such that $E \cdot (-\Kxt) = 1$. In particular, if $d \neq 1$ then $E$ is a $(-1)$curve   
\end{lem}
\begin{proof}
	The first statement holds since $-\Kxt$ is nef. The second half follows from Lemma {\ref{Lem:exccurves}}.
\end{proof}

\begin{lem}\label{Lem3.3}
	Let $D_1$ and $D_2$ be pre$(-1)$curves on $\Xt$. If $D_1$ and $D_2$ are not linearly equivalent, then $D_1 \cdot D_2 \geq 0$.
\end{lem}
\begin{proof}
	If $\dim H^0(\Xt ,\sh{\Xt}(D_1 - D_2)) > 0$, then there exists an effective divisor $\mF$ which linearly equivalent to $(D_1 - D_2)$. Since $\mF \cdot \Kxt = 0$ by Lemma {\ref{Lem:exccurves}} (4), $\mF$ is the sum of $(-2)$curves. Therefore $(\mF)^2 \leq -2$. This means $(D_1 - D_2)^2 \leq -2$ and thus $D_1 \cdot D_2 \geq 0$.

	Suppose that $\dim H^0(\Xt ,\sh{\Xt}(D_1 - D_2)) = 0$. Since $(\Kxt - D_1 + D_2) \cdot (-\Kxt) = -d < 0$, we have $\dim H^0(\Xt, \sh{\Xt}(\Kxt - D_1 + D_2)) = 0$. Thus, by Riemann-Roch theorem, we have $(D_1 - D_2) \cdot (D_1 - D_2 - \Kxt)/2 + 1 \leq 0$. This means $D_1 \cdot D_2 \geq 0$.
\end{proof}

\begin{prop}\label{Prop3.4}
	Let $D$ be a pre$(-1)$curve on $\Xt$. $D$ is not a $(-1)$curve if and only if there exists a $(-2)$curve $F$ such that $D \cdot F = -1$.
\end{prop}
\begin{proof}
	If there exists a $(-2)$curve $F$ such that $D \cdot F = -1$, $F$ is contained in $D$. Thus $D$ is not a $(-1)$curve. Conversely, suppose that $D$ is not a $(-1)$curve. By Lemma {\ref{Lem3.2}}, there exists a prime divisor $E$ contained in $D$ such that $E \cdot (-\Kxt ) = 1$. Then $\Dp := D - E$ satisfies $\Dp \cdot \Kxt = 0$. By Lemma {\ref{Lem:exccurves}} (4), $\Dp$ is the sum of $(-2)$curves and thus $(\Dp)^2 = (D - E)^2 = -2 - 2D \cdot E \leq -2$. This means $D \cdot E \geq 0$. Since $D \cdot \Dp = D \cdot (D - E) = -1 - D \cdot E \leq -1$, there is a $(-2)$curve $F$ contained in $\Dp$ such that $D \cdot F \leq -1$. On the other hand, intersection number between a pre$(-1)$curve and a $(-2)$curve is at least $-1$ by the following lemma. Thus $D \cdot F = -1$. This $F$ is the desired $(-2)$curve.
\end{proof}

\begin{lem}
	Let $D$ be a pre$(-1)$curve on $\Xt$ and $F$ be a $(-2)$curve on $\Xt$. Then $D \cdot F \geq -1$.
\end{lem}
\begin{proof}
$\dim H^0 (\Xt, \sh{\Xt}(-D + F)) = 0$ since $(-D + F) \cdot (-\Kxt) = -1$ and $-\Kxt$ is nef. On the other hand, since $(\Kxt + D - F) \cdot (-\Kxt) = -d + 1$, we have either $\dim H^0 (\Xt, \sh{\Xt}(\Kxt + D - F)) = 0$ or $d = 1$ and $(\Kxt + D - F)$ is linearly equivalent to a sum of $(-2)$curves. If the later holds, we have $(\Kxt + D - F)^2 < 0$ and thus $D \cdot F > 0$. 

	We assume that $\dim H^0 (\Xt, \sh{\Xt}(\Kxt + D - F)) = 0$. Then, by Riemann-Roch theorem, we have $(-D + F) \cdot (-D + F - \Kxt)/2 + 1 \leq 0$. This means $D \cdot F \geq -1$.
\end{proof}
As discussed in Lemma {\ref{Lem:intersec_No}}, if $d = 2$ then $D \cdot F = -1, 0$ or $1$. When $D \cdot F = -1$ or $1$, the following lemma holds.

\begin{lem}{\label{Lem:minusF}}
	Let $D$ be a pre$(-1)$curve on $\Xt$ and $F$ be a $(-2)$curve on $\Xt$. If $D\cdot F = -1$, then $(D - F)$ is also a pre$(-1)$curve. If $D \cdot F = 1$, then $(D + F)$ is also a pre$(-1)$curve.
\end{lem}
\begin{proof}
	If $D \cdot F = -1$, then $(D - F)^2 = -1 +2 -2 = -1$ and $(D - F) \cdot \Kxt = -1$. Since $(D - F)$ is an effective divisor, $(D - F)$ is a pre$(-1)$curve. If $D \cdot F = 1$, then $(D + F)^2 = -1 +2 -2 = -1$ and $(D + F) \cdot \Kxt = -1$. Since $(D + F)$ is an effective divisor, $(D + F)$ is a pre$(-1)$curve.
\end{proof}

	By using Proposition {\ref{Prop3.4}} and Lemma {\ref{Lem:minusF}} repeatedly, we obtain the following corollary.
\begin{cor}{\label{Cor3.5}}
	If a divisor $D$ on $\Xt$ satisfies $D^2 = D \cdot \Kxt = -1$, then the complete linear system $|D|$ consists of a single pre$(-1)$curve.
\end{cor}
\begin{proof}
	Since $(\Kxt - D) \cdot (-\Kxt) = - d - 1 < 0$, we have $\dim H^0(\Xt , \sh{\Xt}(\Kxt - D)) = 0$. Thus $\dim H^0(\Xt , \sh{\Xt}(D)) \geq (D \cdot (D - \Kxt))/2 + 1 = 1$. This means there exists a pre$(-1)$curve which is linearly equivalent to $D$.

	Suppose that $D_1$ and $D_2$ are pre$(-1)$curves which are linearly equivalent to $D$. Let $E_1$ be the prime divisor contained in $D_1$ such that $E_1 \cdot (-\Kxt) = 1$ and let $\Dp_1 := D_1 - E_1$. We show $D_1 = D_2$ by induction on $n$, where $n$ is the number of components of $\Dp_1$

	If $n = 0$, $D_1$ is a $(-1)$curve. By Proposition {\ref{Prop3.4}}, we have $D_1 \cdot F = D_2 \cdot F \geq 0$ for all $(-2)$curve $F$. This means $D_2$ is also a $(-1)$curve. Thus $D_1 = D_2$. If $n > 0$, there is a $(-2)$curve $F$ such that $D_1 \cdot F = -1$ by Proposition {\ref{Prop3.4}}. Since $D_1$ is linearly equivalent to $D_2$, we have $D_2 \cdot F = -1$. By Lemma {\ref{Lem:minusF}}, $(D_1 - F)$ and $(D_2 - F)$ are pre$(-1)$curves. By the induction hypothesis, we have $D_1 - F = D_2 - F$. This means $D_1 = D_2$.
\end{proof}
\subsection{The ($-$1)curve passing through two singular points}{\label{subsection:two pts}}
	In the rest of Section 3, let $X$ be a RDP Del Pezzo surface of degree $2$ over a perfect field $k$ and $\Xt$ be the weak Del Pezzo surface defined as the minimal resolution of $X$. In this subsection, we show that if $X \times_k \kb$ has two singular points, $\Xt$ has one or two $(-1)$curves which passing through the two singularities on $X$. By collapsing these $(-1)$curves, $\Xt$ is non-minimal if $X$ has just two singular points over $\kb$ (Theorem {\ref{Thm3.11}}).

	For Proposition {\ref{Prop3.7}} and {\ref{Prop3.8}}, we need some notations. By Proposition {\ref{Prop:HW}}, there exists a blowing up $\pi :\Xt \times_k \kb \rightarrow \bP^2_{\kb}$ with center $\Sigma = \{ x_1,\cdots ,x_7\}$ over $\kb$. Then the Picard group of $\Xt \times_k \kb$ is generated by the class of inverse image of lines, denote this class by $l_0$, and the classes of inverse image of $x_i$, denoted by $l_i$ $(1 \leq i \leq 7)$. The intersection numbers of these are $l_0 \cdot l_0 = 1$, $l_i \cdot l_i = -1$ $(1\leq i \leq 7)$ and $l_i \cdot l_j = 0$ $(0\leq i,j \leq 7)$.

	We list all pre$(-1)$curves. Let $\mL := a_0 l_0 + a_1 l_1 + \cdots + a_7 l_7$, where $a_i \in \bZ$ and $a_0 \geq 0$ since there is no effective divisor corresponding to $\mL$ if $a_0 < 0$. Then, $(\mL)^2 = \mL \cdot \Kxt = -1$ if and only if $a_0^2 - a_1^2 - a_2^2 - \cdots - a_7^2 = -1$ and $-3a_0 - a_1 -a_2 - \cdots -a_7 = -1$. On the other hand, we have $(a_1 + a_2 + \cdots + a_7)^2 \leq 7(a_1^2 + a_2^2 + \cdots + a_7^2)$ by Cauchy-Schwartz inequality. Therefore, $(-3a_0 + 1)^2 \leq 7(a_0^2 + 1)$ and thus $0 \leq a_0 \leq 3$. From the above, if $(\mL)^2 = \mL \cdot \Kxt = -1$ then $\mL$ is one of the following 56 classes.
\begin{itemize}
	\item[(a)] $\mA_i := l_i$,
	\item[(b)] $\mB_{ij} := l_0 - l_i - l_j$,
	\item[(c)] $\mC_{ij} := 2l_0 - l_1 - \cdots - l_7 + l_i + l_j$ or
	\item[(d)] $\mD_i := 3l_0 - l_1 - \cdots - l_7 -l_i$.
\end{itemize}
By Corollary {\ref{Cor3.5}}, these each 56 elements corresponds to a unique pre$(-1)$curve.

\begin{cor}{\label{Cor:56}}{\label{Cor:56pre-1curves}}
	A weak Del Pezzo surface of degree $2$ over $k$ has just $56$ pre$(-1)$curves.
\end{cor}
This corollary is used in Section 5.

	Similarly, if $(\mL)^2 = -2$ and $\mL \cdot \Kxt = 0$, then $(0 \leq) a_0 \leq 2$. Thus $\mL$ is one of the followings:
\begin{itemize}
	\item[(a$^{\prime}$)] $\mAp_{i,j} := l_i - l_j$,
	\item[(b$^{\prime}$)] $\mBp_{ijk} := l_0 - l_i - l_j - l_k$ or
	\item[(c$^{\prime}$)] $\mCp_{i} := 2l_0 - l_1 - \cdots - l_7 + l_i$.
\end{itemize}
However, unlike Corollary {\ref{Cor:56}}, these classes may not correspond to a $(-2)$curve. For example, if $\mA_{i,j}$ has a $(-2)$curve, $\mA_{j,i}$ does not have a $(-2)$curve.

\begin{table}[p]
	\caption{The intersection number of two pre$(-1)$curves}
	\begin{tabular}{|c|cl|cl|cl|cl|}\hline
		 & \multicolumn{2}{|c|}{$\mA_k$}& \multicolumn{2}{|c|}{$\mB_{kl}$}& \multicolumn{2}{|c|}{$\mC_{kl}$} & \multicolumn{2}{|c|}{$\mD_{k}$}\\ \hline

		              &$ -1 $&$ (i=k) $&$ 1 $&$ (i\in \{ k,l\} ) $&$ 0 $&$ (i\in \{ k,l\} ) $&$ 2 $&$ (i=k) $\\ 
		$\mA_{i}$ &$ 0 $&$ (i\neq k) $&$ 0 $&$ (i\notin \{ k,l\} ) $&$ 1 $&$ (i\notin \{ k,l\} ) $&$ 1 $&$ (i\neq k) $\\ \hline

		               &  &  & $-1$ & $(|\{ i,j\}\cap\{ k,l\}| =2)$ & $2$ & $(|\{ i,j\}\cap\{ k,l\}| =2)$ & $0$ & $(k \in \{ i,j\})$\\ 
		$\mB_{ij}$ &  &  & $0$ &  $(|\{ i,j\}\cap\{ k,l\}| =1)$ & $1$ &  $(|\{ i,j\}\cap\{ k,l\}| =1)$ & $1$ & $(k \notin \{ i,j\})$\\ 
		               &  &  & $1$ & $(\{ i,j\}\cap\{ k,l\} =\phi )$ & $0$ & $(\{ i,j\}\cap\{ k,l\} =\phi )$ &  & \\ \hline

		               &  &  &  &  & $-1$ & $(|\{ i,j\}\cap\{ k,l\}| =2)$ & $1$ & $(k \in \{ i,j\})$\\ 
		$\mC_{ij}$ &  &  &  &  & $0$ & $(|\{ i,j\}\cap\{ k,l\}| =1)$ & $0$ & $(k \notin \{ i,j\})$\\ 
		               &  &  &  &  & $1$ & $(\{ i,j\}\cap\{ k,l\} =\phi )$ &  & \\ \hline

		              &  &  &  &  &  &  & $-1$ & $(i=k)$ \\ 
		 $\mD_{i}$&  &  &  &  &  &  & $0$ & $(i\neq k)$\\ \hline
 	\end{tabular}
	\label{Table1}

	\caption{The intersection number of two $(-2)$classes}
	\begin{tabular}{|c|cl|cl|cl|}\hline
		 & \multicolumn{2}{|c|}{$\mAp_{l,m}$}& \multicolumn{2}{|c|}{$\mBp_{lmn}$}& \multicolumn{2}{|c|}{$\mCp_{l}$}\\ \hline

	 & $-2$ & $i = l, j = m$ & $0$ & $\{ i,j\} \cap \{ l,m,n\} = \{ i,j\}$ & $-1$ & $i = l$ \\
	 & $2$ & $i = m, j = l$ & $1$ & $\{ i,j\} \cap \{ l,m,n\} = \{ i\}$ & $1$ & $j = l$ \\
	 & $-1$ & $i = l, j \neq m$ & $-1$ & $\{ i,j\} \cap \{ l,m,n\} = \{ j\}$ & $0$ & $l \notin \{ i,j\}$ \\
	$\mAp_{i,j}$ & $1$ & $i = m, j \neq l$ & $0$ & $\{ i,j\} \cap \{ l,m,n\} = \phi$ &  &  \\
	 & $-1$ & $i \neq l,j = m$ &  &  & &  \\
	 & $0$ & $\{ i,j\} \cap \{ l,m\} = \phi$ &  &  &  & \\ \hline

	 &  &  & $-2$ & $|\{ i,j,k\} \cap \{ l,m,n\}| = 3$ & $0$ & $l \in \{ i,j,k\}$ \\
	 &  &  & $-1$ & $|\{ i,j,k\} \cap \{ l,m,n\}| = 2$ & $-1$ & $l \notin \{ i,j,k\}$ \\
	$\mBp_{ijk}$ &  &  & $0$ & $|\{ i,j,k\} \cap \{ l,m,n\}| = 1$ &  &  \\
	 &  &  & $1$ & $\{ i,j,k\} \cap \{ l,m,n\} = \phi$ &  &  \\ \hline

	 &  &  &  &  & $-2$ & $i = l$ \\
	$\mCp_{i}$ &  &  &  &  & $-1$ & $i \neq l$ \\ \hline
 	\end{tabular}
	\label{Table2}

	\caption{The intersection number of pre$(-1)$curve and $(-2)$class}
	\begin{tabular}{|c|cl|cl|cl|}\hline
		 & \multicolumn{2}{|c|}{$\mAp_{k,l}$}& \multicolumn{2}{|c|}{$\mBp_{klm}$}& \multicolumn{2}{|c|}{$\mCp_{k}$}\\ \hline

	 & $-1$ & $i = k$ & $1$ & $i \in \{ k,l,m\}$ & $0$ & $i = k$\\
	$\mA_{i}$ & $1$ & $i = l$ & $0$ & $i \notin \{ k,l,m\}$ & $1$ & $i \neq k$\\
	 & $0$ & $i \notin \{ k,l\}$ &  &  &  & \\ \hline

	 & $0$ & $\{ i,j\} \cap \{ k,l\} = \{ k,l\}$ & $-1$ & $|\{ i,j\} \cap \{ k,l,m\}| = 2$ & $1$ & $k \in \{ i,j\}$ \\
	$\mB_{ij}$ & $1$ & $\{ i,j\} \cap \{ k,l\} = \{ k\}$ & $0$ & $|\{ i,j\} \cap \{ k,l,m\}| = 1$ & $0$ &  $k \notin \{ i,j\}$\\
	 & $-1$ & $\{ i,j\} \cap \{ k,l\} = \{ l\}$ & $1$ & $\{ i,j\} \cap \{ k,l,m\} = \phi$ &  & \\
	 & $0$ & $\{ i,j\} \cap \{ k,l\} = \phi$ &  &  &  & \\ \hline

	 & $0$ & $\{ i,j\} \cap \{ k,l\} = \{ k,l\}$ & $1$ & $|\{ i,j\} \cap \{ k,l,m\}| = 2$ & $-1$ & $k \in \{ i,j\}$ \\
	$\mC_{ij}$ & $-1$ & $\{ i,j\} \cap \{ k,l\} = \{ k\}$ & $0$ & $|\{ i,j\} \cap \{ k,l,m\}| = 1$ & $0$ &  $k \notin \{ i,j\}$\\
	 & $1$ & $\{ i,j\} \cap \{ k,l\} = \{ l\}$ & $-1$ & $\{ i,j\} \cap \{ k,l,m\} = \phi$ &  & \\
	 & $0$ & $\{ i,j\} \cap \{ k,l\} = \phi$ &  &  &  & \\ \hline

	 & $1$ & $i = k$ & $-1$ & $i \in \{ k,l,m\}$ & $0$ & $i = k$\\
	$\mD_{i}$ & $-1$ & $i = l$ & $0$ & $i \notin \{ k,l,m\}$ & $-1$ & $i \neq k$\\
	 & $0$ & $i \notin \{ k,l\}$ &  &  &  & \\ \hline
 	\end{tabular}
	\label{Table3}
\end{table}

	The intersection number of them is given by Table {\ref{Table1}},{\ref{Table2}} and {\ref{Table3}}. By these tables, the following lemma holds.
\begin{lem}{\label{Lem:intersec_No}}
	Let $D, D_1, D_2$ be pre$(-1)$curves on $\Xt$ and let $F, F_1, F_2$ be $(-2)$curves o $\Xt$. Then, the intersection numbers of these are
\begin{itemize}
	\item[(1)] $D_1 \cdot D_2 = -1, 0, 1$ or $2$
	\item[(2)] $D \cdot F = -1, 0$ or $1$
	\item[(3)] $F_1 \cdot F_2 = -2, 0$ or $1$
\end{itemize}
\end{lem}
\begin{proof}
	(1) and (2) are given by Table {\ref{Table1}} and {\ref{Table3}}, respectively. In Table {\ref{Table2}}, there are $(-2)$classes $\mL_1$ and $\mL_2$ such that $\mL_1 \cdot \mL_2 = -1$ or $2$. However, it is impossible both of such $\mL_1$ and $\mL_2$ have $(-2)$curves.
\end{proof}

When $D_1 \cdot D_2 = 2$, the following lemma holds.
\begin{lem}\label{Lem3.6}
	Let $D_1$ and $D_2$ be pre$(-1)$curves on $\Xt$. Then
	$$D_1 \cdot D_2 = 2 \Leftrightarrow (D_1 + D_2) \in |-\Kxt|.$$
\end{lem}
\begin{proof}
	If $(D_1 + D_2) \in |-\Kxt|$, then $(D_1 + D_2)^2 = -2 + 2 D_1\cdot D_2 = 2$. Thus $D_1 \cdot D_2 = 2$. Conversely, suppose that $D_1 \cdot D_2 = 2$. Since $(-\Kxt - D_1)^2 = (-\Kxt - D_1) \cdot \Kxt = -1$, there exists a pre$(-1)$curve which is linearly equivalent to $(-\Kxt - D_1)$ by Corollary {\ref{Cor3.5}}. Since $(-\Kxt - D_1) \cdot D_2 = -1$, we have $D_2 \in |-\Kxt - D_1|$ by Lemma {\ref{Lem3.3}}.
\end{proof}

\begin{prop}\label{Prop3.7}
	Let $F$ be a $(-2)$curve on $\Xt$. Then pre$(-1)$curves on $\Xt$ satisfy the following properties:
\begin{itemize}
	\item[(1)] The cardinality of $\{ D : pre(-1)curve\ | D \cdot F = 1 \}$ is $12$.
\end{itemize}
Let $D_1$ and $D_2$ be two pre$(-1)$curves such that $D_1 \cdot F = D_2 \cdot F = 1$.
\begin{itemize}
	\item[(2)] $D_1 \cdot D_2 \leq 1$.
	\item[(3)] $D_1 \cdot D_2 = 1$ if and only if $(D_1 + D_2) \in |-\Kxt - F|$. Moreover, for every $D_1 \in \{ D : pre(-1)curve\ | D \cdot F = 1 \}$, there exists a unique $(-2)$curve $D_2$ such that $D_2 \cdot F = 1$ and $D_1 \cdot D_2 = 1$. 
\end{itemize}
\end{prop}
\begin{proof}
	We can assume that $F$ is either $\mAp_{1,2}$ , $\mBp_{123}$ or $\mCp_1$. By Table {\ref{Table3}}, we have 
\begin{itemize}
	\item[]$\{ D\ | D \cdot \mAp_{1,2} = 1 \} = \{ \mA_2 , \mB_{1i} (i = 3,\cdots,7) , \mC_{2i} (i = 3,\cdots,7), \mD_1\} $
	\item[]$\{ D\ | D \cdot \mBp_{123} = 1 \} = \{ \mA_l (l = 1,2,3) ,\mB_{ij} (i,j = 4,\cdots ,7), \mC_{lm} ( l,m = 1,2,3) \} $
	\item[]$\{ D\ | D \cdot \mCp_{1} = 1 \} = \{ \mA_{i} (i = 2,\cdots,7), \mB_{1i} (i = 2,\cdots,7)\} $.
\end{itemize}
Thus (1) holds. Since $F \cdot \Kxt = 0$, (2) follows from Lemma {\ref{Lem3.6}}.
\par
	We show $(3)$. If $(F + D_1 + D_2) \in |-\Kxt|$, we have $(F+ D_1 + D_2)^2 = 2 D_1 \cdot D_2 = 2$. Thus $D_1 \cdot D_2 = 1$. Conversely, suppose that $D_1 \cdot D_2 = 1$. Since $(-\Kxt - D_1 - F)^2 = (-\Kxt - D_1 - F) \cdot \Kxt = -1$, there exists a pre$(-1)$curve which is linearly equivalent to $(-\Kxt - D_1 - F)$ by Corollary {\ref{Cor3.5}}. Since $(-\Kxt - D_1 - F) \cdot D_2 = -1$, we have $D_2 \in |-\Kxt - D_1 - F|$ by Lemma {\ref{Lem3.3}}. The last statement follows from Corollary {\ref{Cor3.5}}.
\end{proof}

\begin{prop}{\label{Prop3.8}}
	Let $F_1$ and $F_2$ be two $(-2)$curves on $\Xt$ such that $F_1 \cdot F_2 = 0$. Then pre$(-1)$curves on $\Xt$ satisfy the following properties:
\begin{itemize}
	\item[(1)] The cardinality of $\{ D : pre(-1)curve\ | D \cdot F_1 = D \cdot F_2 = 1 \}$ is $2$.
\end{itemize}
Let $D_1$ and $D_2$ be the two pre$(-1)$curves such that $D_i \cdot F_1 = D_i \cdot F_2 = 1$ for $i = 1,2$.
\begin{itemize}
	\item[(2)] $D_1 \cdot D_2 = 0$.
	\item[(3)] $(D_1 + D_2) \in |-\Kxt - F_1 - F_2|$.
\end{itemize}
\end{prop}
\begin{proof}
	By Table {\ref{Table2}}, we can assume that $(F_1, F_2)$ is either $(\mAp_{1,2} , \mAp_{3,4})$, $(\mAp_{1,2} , \mBp_{123})$, $(\mAp_{1,2} , \mBp_{345})$, $(\mAp_{1,2} , \mCp_{3})$, $(\mBp_{123} , \mBp_{145})$ or $(\mBp_{123} , \mCp_{1})$.

	If $(F_1,F_2)$ is $(\mAp_{1,2} , \mAp_{3,4})$, the pre$(-1)$curves are $\mB_{13}$ and $\mC_{24}$. 
\par
	If $(F_1,F_2)$ is $(\mAp_{1,2} , \mBp_{123})$, the pre$(-1)$curves are $\mA_{2}$ and $\mC_{23}$. 
\par
	If $(F_1,F_2)$ is $(\mAp_{1,2} , \mBp_{345})$, the pre$(-1)$curves are $\mB_{16}$ and $\mB_{17}$. 
\par
	If $(F_1,F_2)$ is $(\mAp_{1,2} , \mCp_{3})$, the pre$(-1)$curves are $\mA_{2}$ and $\mB_{13}$. 
\par
	If $(F_1,F_2)$ is $(\mBp_{123} , \mBp_{145})$, the pre$(-1)$curves are $\mA_{1}$ and $\mB_{67}$. 
\par
	If $(F_1,F_2)$ is $(\mBp_{123} , \mCp_{1})$, the pre$(-1)$curves are $\mA_{2}$ and $\mA_{3}$. 
\par
	Thus $(1)$ holds. 

	Let $D_1,D_2$ be the two pre$(-1)$curves which intersect $F_1$ and $F_2$. Since $(D_1 + D_2 + F_1) \cdot F_2 = 2$, we have $(D_1 + D_2 + F_1) \notin |-\Kxt|$. By Proposition {\ref{Prop3.7}}, we have $D_1 \cdot D_2 = 0$. $(2)$ holds.

	Since $D_1 \cdot D_2 = 0$, we have $(-\Kxt - D_1 -F_1 - F_2) \cdot D_2 = -1$. On the other hand, $(-\Kxt - D_1 -F_1 - F_2)$ is linearly equivalent to a pre$(-1)$curve since $(-\Kxt - D_1 -F_1 - F_2)^2 = (-\Kxt - D_1 -F_1 - F_2) \cdot \Kxt = -1$. By Lemma {\ref{Lem3.3}}, $D_2 \in |-\Kxt - F_1 - F_2 - D_1|$.
\end{proof}

	Let $F_1 , F_2 ,D_1$ and $D_2$ be as Proposition {\ref{Prop3.8}}. By Lemma {\ref{Lem3.2}}, each $D_i$ contains a $(-1)$curve $E_i$. We consider the case where $f(F_1)$ and $f(F_2)$ are two distinct singular points on $X$, where $f : \Xt \rightarrow X$ is a minimal resolution of $X$. 

\begin{prop}{\label{Prop3.9}}
	Let $E_1,E_2$ be as above. If $f(F_1)$ and $f(F_2)$ are distinct two singular points on $X$, we have $E_1 \cdot E_2 = 0$ or $-1$.
\end{prop}
\begin{proof}
	Since $(D_1 + D_2 + F_1 + F_2) \in |-\Kxt|$, we have $(E_1 + E_2 + \mF) \in |-\Kxt|$ where $\mF$ is the sum of some $(-2)$curves containing $F_1$ and $F_2$. Then
	$$-2 + 2 E_1 \cdot E_2 = (E_1 + E_2)^2 = (-\Kxt - \mF)^2 = 2 + (\mF)^2$$
Since $F_1, F_2 \in \mF$ and $f(F_1) \neq f(F_2)$, $\mF$ has at least two connected components. This induces $(\mF)^2 \leq -4$. Thus $E_1 \cdot E_2 \leq 0$.
\end{proof}

	In practice, the assumption $f(F_1) \neq f(F_2)$ is not necessary (cf. Proposition {\ref{Prop:MainProp3.5}}). This is proved in Section 3.5. 

	Let $p_1$ and $p_2$ be two distinct singular points on $X$ over $\kb$. By Corollary {\ref{Cor:RDPDP<->weakDP}} (2), $f^{-1}(p_1)$ and $f^{-1}(p_2)$ are the sum of some $(-2)$curves. If $(-2)$curves $F$ and $\Fp$ satisfy $f(F) = f(\Fp) = p_1$, then there exist $(-2)$curves $F_1, F_2, \cdots, F_n$ such that $F_1 = F, F_n = \Fp$ and $F_i \cdot F_{i + 1} = 1$.
$$
\xygraph{
	\circ ([]!{+(0,+.2)} {F = F_1}) - [r]
	\circ ([]!{+(0,+.2)} {F_2}) - [r]
	\cdots - [r]
	\circ ([]!{+(0,+.2)} {F_{n - 1}}) - [r]
	\circ ([]!{+(0,+.2)} {\Fp = F_n})
}
$$
Let $F, \Fp, G$ and $\Gp$ be $(-2)$curves on $\Xt$ such that $f(F) = f(\Fp) = p_1$ and $f(G) = f(\Gp) = p_2$. Then the following holds:

\begin{prop}{\label{Prop3.10}}
	 Let $D_1, D_2$ be the two pre$(-1)$curves which intersect $F$ and $G$ and let $E_1$ and $E_2$ be the $(-1)$curves contained in $D_1$ and $D_2$, respectively (such pre$(-1)$curves and $(-1)$curves exist by Proposition {\ref{Prop3.8}} and Lemma {\ref{Lem3.2}}). Similarly, $\Dp_1, \Dp_2, \Ep_1$ and $\Ep_2$ are defined by $\Fp$ and $\Gp$. Then, $\{ E_1,E_2 \} = \{ \Ep_1, \Ep_2 \}$ holds.
\end{prop}
\begin{proof}
	Since $f(F) = f(\Fp)$ and $f(G) = f(\Gp)$, there exist sequences of $(-2)$curves $\{ F_1, F_2, \cdots, F_n\}$ such that $F_1 = F$, $F_n = \Fp$ and $F_i \cdot F_{i+1} = 1$ and $\{ G_1, G_2, \cdots, G_m \}$ such that $G_1 = G$, $G_n = \Gp$ and $G_i \cdot G_{i+1} = 1$. Then, to show that $(F, G)$ and $(\Fp, \Gp)$ define the same $(-1)$curves, it suffices to show the case where $F = \Fp$ and $G \cdot \Gp = 1$. Indeed, If this case is proved, we can show that $(F_1, G_1)$ and $(F_1, G_2)$ define the same two $(-1)$curves. Further, $(F_1, G_2)$ and $(F_1, G_3)$ define the same $(-1)$curves. By repeating this, $(F_1,G_1)$ and $(F_1, G_m)$ define the same $(-1)$curves.
$$
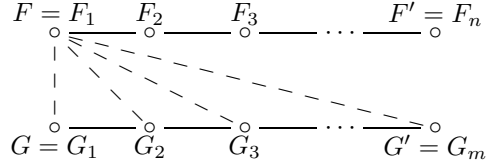
\xygraph{!~:{@{--}}
	\circ ([]!{+(0,+.2)} {F = F_1}) (
	- [r] \circ ([]!{+(0,+.2)} {F_2})
	- [r] \circ ([]!{+(0,+.2)} {F_3})
	- [r] \cdots
	- [r] \circ ([]!{+(0,+.2)} {\Fp = F_n}),
	: [d] \circ ([]!{+(0,-.2)} {G = G_1})
	- [r] \circ ([]!{+(0,-.2)} {G_2})
	- [r] \circ ([]!{+(0,-.2)} {G_3})
	- [r] \cdots
	- [r] \circ ([]!{+(0,-.2)} {\Gp = G_m}),
	: [rd] \ ,
	: [rrd] \ ,
	: [rrrrd] \ 
)
}
$$
Similarly, $(F_1, G_m), (F_2, G_m), (F_3, G_m), \cdots, (F_n, G_m)$ define the same two $(-1)$curves. Thus, $(F_1, G_1)$ and $(F_n, G_m)$ define the same two $(-1)$curves.

	Suppose that $F = \Fp$ and $G \cdot \Gp = 1$. Then $(D_1 + D_2) \cdot \Gp = -1$ since $(D_1 + D_2) \in |-\Kxt - F - G|$ and $G \cdot \Gp = 1$. This means $D_1$ or $D_2$ contains $\Gp$ (assume $\Gp \subset D_1$). Similarly, $(\Dp_1 + \Dp_2) \cdot G = -1$ and we can assume that $G \subset \Dp_1$. By Lemma {\ref{Lem:minusF}}, $(D_1 - \Gp)$ is a pre$(-1)$curve. Then, $(D_1 - \Gp)$ satisfies $(D_1 - \Gp) \cdot F = (D_1 - \Gp) \cdot \Gp = 1$ and $(D_1 - \Gp) \cdot G = 0$. This means $D_1 - \Gp = \Dp_2$ and $E_1 = \Ep_2$. Similarly, $\Dp_1 - G = D_2$ and thus $\Ep_1 = E_2$. Thus $\{ E_1,E_2 \} = \{ \Ep_1, \Ep_2 \}$.
\end{proof}
	Proposition {\ref{Prop3.9}} and {\ref{Prop3.10}} mean that two singular points define two $(-1)$curves on $\Xt$ which intersect or a single $(-1)$curve on $\Xt$. The following lemma states that such two $(-1)$curves pass through the two singularities.
\begin{lem}{\label{Lem3.12}}
	Let $p$ be a singular point of $\Xkb$ and let $\mF$ be the sum of $(-2)$curves collapsed by $f$ to $p$ (that is, $\mF = f^{-1}(p)$). For a $(-2)$curve $F$ contained in $\mF$ and a pre$(-1)$curve $D$, if $D \cdot F = 1$ then the $(-1)$curve $E \subset D$ intersects one of $(-2)$curves in $\mF$.
\end{lem}
\begin{proof}
	Let $n$ be the number of components of $D$. We show this lemma by induction on $n$.

	If $n = 1$, then $D = E$ and $E$ intersects $F \subset \mF$. Suppose that $n > 1$. By Proposition {\ref{Prop3.4}}, there exists a $(-2)$curve $\Fp$ such that $\Fp \cdot D = -1$. By Lemma {\ref{Lem:minusF}}, $(D - \Fp)$ is a pre$(-1)$curve. We show that $(D - \Fp)$ intersects $\mF$. If $F \cdot \Fp = 0$, we have $(D - \Fp) \cdot F = 1$. If $F \cdot \Fp = 1$, we have $(D - \Fp) \cdot \Fp = 1$ and $\Fp \subset \mF$. Thus $(D - \Fp)$ intersects $\mF$ in any cases. By induction hypothesis, the $(-1)$curve $\subset D - \Fp$, that is $E$, intersects $\mF$.
\end{proof}
	By collapsing the $(-1)$curve(s) defined by two singularities, we have the following theorem:
\begin{thm}{\label{Thm3.11}}
	Let $X$ be a RDP Del Pezzo surface of degree $2$ over a perfect field $k$ and $\Xt$ be a minimal resolution of $X$. If $X \times_k \kb$ has just two singular points $p_1, p_2$, then there exists a birational morphism $\Xt \rightarrow \Vt$, where $\Vt$ is a weak Del Pezzo surface of degree $3$ or $4$. In particular, $X$ is $k$-unirational if $\Xt$ has a $k$-point.
\end{thm}
\begin{proof}
	Let $f : \Xt \rightarrow X$ be the minimal resolution of $X$ and let $\mF_i$ be $f^{-1}(p_i)$ $( i = 1,2)$. Then two $(-1)$curves $E ,\Ep$ on $\Xt$ are defined by $\mF_1, \mF_2$ and  Proposition {\ref{Prop3.10}}. Since $X$ has only two singular points, $(E + \Ep)$ is defined over $k$. By Proposition {\ref{Prop3.9}}, $E \cdot \Ep = 0$ or $E = \Ep$. If $E \cdot \Ep = 0$, $(E + \Ep)$ can be collapsed. Thus we obtain a birational morphism to a weak Del Pezzo surface of degree $4$. If $E = \Ep$, $E$ is defined over $k$. Thus $X$ is birationally equivalent to a weak Del Pezzo surface of degree $3$. 

	The last statement follows from Proposition {\ref{Prop:bir degree3}} and {\ref{Prop:bir degree4}}.
\end{proof}

	In practice, $\Vt$ can be classified according to whether the type of singularities on $X$ is $2A_1, A_1 + A_2, A_1 + A_3, A_1 + A_4, A_1 + A_5, A_1 + D_4, A_1 + D_5, A_1 + D_6,  2A_2, A_2 + A_3, A_2 + A_4, A_2 + A_5$ or $2A_3$. This allows us to classify the unirationality of $X$ in more detail. See Appendix A.
	
\subsection{The case with three singular points}
	Let $X$ and $\Xt$ be as above and suppose that $X$ has three singular points. By Proposition {\ref{Prop3.10}}, these three singular points define 1 to 6 $(-1)$curves. We show these $(-1)$curves do not intersect.
\begin{lem}{\label{Lem3.13}}
	Let $\mF, \mF_1, \mF_2$ be inverse images of distinct three singular points on $X$ by $f$. Let $E_i$ be one of the $(-1)$curve intersecting $\mF$ and $\mF_i$ $( i = 1,2)$. Then $E_1 \cdot E_2 = 0$ or $-1$.
\end{lem}
\begin{proof}
	Since $(E_1 + E_2)$ intersects some $(-2)$curves, we have $(E_1 + E_2) \notin |-\Kxt|$. This means $E_1 \cdot E_2 \neq 2$ by Lemma {\ref{Lem3.6}}. Thus, it suffices to show that $E_1 \cdot E_2 \neq 1$. Let $F$ be the component of $\mF$ intersecting $E_1$, let $F_1$ be the component of $\mF_1$ intersecting $E_1$ and let $F_2$ be a component of $\mF_2$. By Proposition {\ref{Prop3.8}} and {\ref{Prop3.10}}, there exists a pre$(-1)$curve $D_2$ such that $F \cdot D_2 = F_2 \cdot D_2 = 1$ and $E_2 \subset D_2$. 

	Suppose that $E_1 \cdot E_2 = 1$. Then $E_1 \cdot D_2 \geq 1$. By Proposition {\ref{Prop3.7}}, $(E_1 + D_2) \in |-\Kxt - F|$ since $E_1 \cdot F = D_2 \cdot F = 1$. Thus $(E_1 + D_2) \cdot F_2 = 0$. This contradicts $D_2 \cdot F_2 = 1$. Thus $E_1 \cdot E_2 \leq 0$.
\end{proof}

\begin{thm}{\label{Thm3.14}}
	Let $X$ be a RDP Del Pezzo surface of degree $2$ over a perfect field $k$ and $\Xt$ be a minimal resolution of $X$. If $X \times_k \kb$ has just three singular points $p_1, p_2, p_3$, then there exists a birational morphism $\Xt \rightarrow \Vt$, where $\Vt$ is a weak Del Pezzo surface of degree at least  $3$. In particular, $X$ is $k$-unirational if $\Xt$ has a $k$-point.
\end{thm}
\begin{proof}
	Let $f : \Xt \rightarrow X$ be the minimal resolution of $X$ and let $\mF_i := f^{-1}(p_i)$ $( i = 1,2,3)$. Let $E_{ij} ,\Ep_{ij}$ be the $(-1)$curves defined by $\mF_i, \mF_j$ and Proposition {\ref{Prop3.10}}. Since $X$ has only three singularities, $(E_{12} + \Ep_{12} + E_{23} + \Ep_{23} + E_{13} + \Ep_{13})$ is defined over $k$. By Proposition {\ref{Prop3.9}}, we have $E_{ij} \cdot \Ep_{ij} = 0$ or $E_{ij} = \Ep_{ij}$ $(i,j = 1,2,3)$. By Lemma {\ref{Lem3.13}}, we have $E_{ij} \cdot E_{ik} = 0$ or $E_{ij} = E_{ik}$ $(i,j,k = 1,2,3)$. Thus the reduced components of $(E_{12} + \Ep_{12} + E_{23} + \Ep_{23} + E_{13} + \Ep_{13})$ can be collapsed and we obtain the desired birational morphism.
\end{proof}
As in Section 3.2, $\Vt$ can be classified according to whether the type of singularities on $X$ is  $3A_1, 2A_1 + A_2, 2A_1 + A_3, 2A_1 + D_4, A_1 + 2A_2, A_1 + A_2 + A_3, A_1 + 2A_3$ or $3A_2$. See Appendix A.
\subsection{The case with four singularities}
	Suppose that $\Xkb$ has just four singularities. Then the type of singularities is either $4A_1, 3A_1 + A_2, 3A_1 + A_3$ or $3A_1 + D_4$ by Proposition {\ref{Prop:singularity type}}. In particular, $\Xt$ has at least three $(-2)$curves which intersect no $(-2)$curve. From now on, a $(-2)$curve $F$ is a {\em{A$_1$($-$2)curve}} if $F \cdot G \neq 1$ for all $(-2)$curves $G$.

\begin{lem}{\label{Lem3.15}}
	Let $F_1, F_2$ be $A_1$ $(-2)$curves on $\Xt$. Let $D, \Dp$ be the two pre$(-1)$curves defined by $F_1$, $F_2$ and Proposition {\ref{Prop3.8}}. Then,
\begin{itemize}
	\item[(a)] at least one of $\{ D, \Dp\}$ is a $(-1)$curve.
	\item[(b)] the followings are equivalent:
		\begin{itemize}
			\item[(i)] $\Dp$ is not a $(-1)$curve.
			\item[(ii)] There exists a $(-2)$curve $F$ such that $F \neq F_1, F_2$ and $F \cdot D = 1$.
		\end{itemize}
\end{itemize}
\end{lem}
\begin{proof}
	By Lemma {\ref{Lem3.2}}, $D$ contains a $(-1)$curve $E$ as a component. By Lemma {\ref{Lem3.12}}, we have $E \cdot F_1 = E \cdot F_2 = 1$. However, the only pre$(-1)$curves intersecting $F_1$ and $F_2$ are $D$ and $\Dp$ by Proposition {\ref{Prop3.8}} (1). Thus $E$ is $D$ or $\Dp$. Therefore, (a) holds.

	We show (b). Suppose that $\Dp$ is not a $(-1)$curve. By Proposition {\ref{Prop3.4}}, there exists a $(-2)$curve $F$ such that $F \cdot \Dp = -1$. On the other hand, $(D + \Dp) \in |-\Kxt - F_1 - F_2|$ by Proposition {\ref{Prop3.8}}. Thus $D \cdot F = 1$. Conversely, suppose that there exists a $(-2)$curve $F \neq F_1, F_2$ such that $F \cdot D = 1$. Then $\Dp \cdot F = -1$ since $(D + \Dp) \in | -\Kxt - F_1 - F_2|$. Thus $\Dp$ is not a $(-1)$curve.
\end{proof}

\begin{definition}
	Let $E$ be a $(-1)$curve on $\Xt$. $E$ is a {\em{($-$1)curve of Lemma {\ref{Lem3.15}}}} if $E$ intersects three $(-2)$curves.
\end{definition}

\begin{lem}{\label{Lem:(-1)curveofLem}}
	Let $E$ be a $(-1)$curve of Lemma {\ref{Lem3.15}} and $F_1, F_2, F_3$ be the three $(-2)$curves on $\Xt$ which intersect $E$. Then, the following statements hold:
\begin{itemize}
	\item[(1)] $2E \in |-\Kxt - F_1 - F_2 - F_3|$,
	\item[(2)] $F_i$ are $A_1$ $(-2)$curves,
	\item[(3)] $E$ intersects no $(-2)$curve except $F_1, F_2$ and $F_3$.
\end{itemize}
\end{lem}
\begin{proof}
	By Proposition {\ref{Prop3.8}}, there exists a pre$(-1)$curve $\Dp$ such that $\Dp \cdot F_1 = \Dp \cdot F_2 = 1$ and $(\Dp + E) \in |-\Kxt - F_1 - F_2|$. Then, $(\Dp + E) \cdot F_3 = -(F_1 + F_2) \cdot F_3 \leq 0$ and thus $\Dp \cdot F_3 \leq -1$. Therefore $\Dp \cdot F_3 = -1$ by Lemma {\ref{Lem:intersec_No}} and $F_1 \cdot F_3 = F_2 \cdot F_3 = 0$ (This is a part of (2)). By Lemma {\ref{Lem:minusF}}, $(\Dp - F_3)$ is a pre$(-1)$curve. Further $\Dp - F_3 = E$ since $(\Dp - F_3) \cdot E = -1$. Thus $2E \in |-\Kxt - F_1 - F_2 - F_3|$.

	We show $G \cdot F_i \leq 0$ for all $(-2)$curves $G$. If $G = F_i$, then it is already proved that $G \cdot F_j = 0$ $(j \neq i)$. Let $G$ be a $(-2)$curve $\neq F_1, F_2, F_3$. Since $2E \cdot G \geq 0$, we have $(F_1 + F_2 + F_3) \cdot G \leq 0$. Thus $F_1 \cdot G = F_2 \cdot G = F_3 \cdot G = 0$. Therefore (2) holds. (3) follows from (1) and (2).
\end{proof}

	The following lemma is also used in Section 3.5.
\begin{lem}{\label{Lem3.16}}
	Suppose that $\Xkb$ has at least four singularities and the type of singularities of $\Xkb$ is not $4A_1$. Let $F_1,F_2,F_3$ be $A_1$ $(-2)$curves on $\Xt$ and let $D_{ij}, \Dp_{ij}$ be the two pre$(-1)$curves defined by $F_i$, $F_j$ and Proposition {\ref{Prop3.8}} $(i,j = 1,2,3)$. Then at least one of $\{ D_{12}, \Dp_{12}, D_{13}, \Dp_{13}, D_{23}, \Dp_{23} \}$ is not a $(-1)$curve.
\end{lem}
\begin{proof}
	Suppose that all of the six pre$(-1)$curves are $(-1)$curves. By Lemma {\ref{Lem3.15}}, $D_{ij}, \Dp_{ij}$ do not intersect with any $(-2)$curves except $F_i$ and $F_j$. Thus $D_{12}$, $\Dp_{12}$, $D_{13}$, $\Dp_{13}$, $D_{23}$ and $\Dp_{23}$ are distinct six $(-1)$curves. Further, these six $(-1)$curves do not intersect with each other by Lemma {\ref{Lem3.13}}. Thus $(D_{12} + \Dp_{12} + D_{13} + \Dp_{13} + D_{23} + \Dp_{23})$ defines a birational morphism $\Xt \times_k \kb \rightarrow \Vt$ over $\kb$. Then $\Vt$ is a weak Del Pezzo surface of degree $8$ over $\kb$. Since $\Xkb$ is not type $4A_1$, $\Xt$ has at least five $(-2)$curves and thus $\Vt$ has at least two $(-2)$curves. However, a weak Del Pezzo surface of degree $8$ has at most one $(-2)$curve. Indeed, a weak Del Pezzo surface is a Hirzebruch surface $\bF_2$ or a smooth Del Pezzo surface by Proposition {\ref{Prop:HW}} and $\bF_2$  has only one $(-2)$curve. Therefore, at least one of $\{ D_{12}, \Dp_{12}, D_{13}, \Dp_{13}, D_{23}, \Dp_{23} \}$ is not a $(-1)$curve.
\end{proof}

	By Lemma {\ref{Lem3.15}} and {\ref{Lem3.16}}, $\Xt$ has a $(-1)$curve of Lemma {\ref{Lem3.15}} except $4A_1$ type. By using this $(-1)$curve, we can show the following proposition.

\begin{prop}{\label{Prop3.17}}
	If singularities of $\Xkb$ is either type $3A_1 + A_2$, $3A_1 + A_3$ or $3A_1 + D_4$, then there exists a birational morphism from $\Xt$ to a weak Del Pezzo surface of degree $3$. In particular, $X$ is $k$-rational.
\end{prop}
\begin{proof}
	Let $F_1, F_2, F_3$ be the three $A_1$ $(-2)$curves. By Lemma {\ref{Lem3.15}} and {\ref{Lem3.16}}, there exists a $(-1)$curve of Lemma {\ref{Lem3.15}} which intersects $F_1, F_2$ and $F_3$, denoted by $E$. Then $E$ is a unique $(-1)$curve of Lemma {\ref{Lem3.15}}. Indeed, if $\Ep$ is a $(-1)$curve of Lemma {\ref{Lem3.15}}, then $2\Ep$ is also contained in $|-\Kxt - F_1 - F_2 - F_3|$ by Lemma {\ref{Lem:(-1)curveofLem}}. Thus $2E \cdot 2\Ep = (-\Kxt - F_1 - F_2 - F_3)^2 = -4$ and $E = \Ep$. Since $E$ is unique, $E$ is defined over $k$ and defines a birational morphism $\Xt \rightarrow \Vt$, where $\Vt$ is a weak Del Pezzo surface of degree $3$ over $k$. Further the surface obtained by collapsing all $(-2)$curves on $\Vt$ is a RDP Del Pezzo surface of degree $3$ with only one singularity.  It is $k$-rational by Proposition {\ref{Prop:bir degree3}}.
\end{proof}

	Even if singularities of $\Xkb$ is type $4A_1$, $\Xt$ may be non-minimal.
\begin{prop}{\label{Prop3.18}}
	If the type of singularities of $\Xkb$ is type $4A_1$ but not $\ctext{3}$ of Theorem {\ref{MainThm1}}, then $\Xt$ is not minimal.
\end{prop}
\begin{proof}
	Under the assumption in the statement, we have the following three cases.
\begin{itemize}
	\item[(1)] $\Xt$ has a $(-1)$curve of Lemma {\ref{Lem3.15}};
	\item[(2)] $\Xt$ has no $(-1)$curve of Lemma {\ref{Lem3.15}} and $\Xt$ has a singular point defined over $k$;
	\item[(3)] $\Xt$ has no $(-1)$curve of Lemma {\ref{Lem3.15}} and the four singularities of $\Xkb$ are conjugate by two.  
\end{itemize}

	Suppose (1). We show that the $(-1)$curve of Lemma {\ref{Lem3.15}} is unique. Let $E_1, E_2$ be $(-1)$curves of Lemma {\ref{Lem3.15}}. Since $\Xt$ has just four $(-2)$curves, there exist at least two $(-2)$curves which intersect with both $E_1$ and $E_2$. Thus $E_1 = E_2$ by Lemma {\ref{Lem3.15}}. Thus the unique $(-1)$curve of Lemma {\ref{Lem3.15}} defines a birational morphism $\Xt \rightarrow \Vt$ where $\Vt$ is a weak Del Pezzo surface of degree $3$.

	Suppose (2). Let $F$ be the $(-2)$curve defined over $k$ and $F_1, F_2, F_3$ be the other $(-2)$curves. Let $D_i, \Dp_i$ be the pre$(-1)$curves which intersects $F$ and $F_i$ $( i = 1,2,3)$. Then $D_i, \Dp_i$ are $(-1)$curves by Lemma {\ref{Lem3.15}}. Since $(D_1 + \Dp_1 + D_2 + \Dp_2 + D_3 + \Dp_3)$ is defined over $k$, we obtain a birational morphism $\Xt \rightarrow \Vt$ where $\Vt$ is a weak Del Pezzo surface of degree $8$.

	Suppose (3). Let $F_1, F_2, F_3$ and $F_4$ be the $(-2)$curves. We can assume that $(F_1 + F_2)$ and $(F_3 + F_4)$ are defined over $k$. Let $D_{12}$ and $\Dp_{12}$ be the two pre$(-1)$curves which intersect $F_1$ and $F_2$. Then $D_{12}$ and $\Dp_{12}$ are $(-1)$curves by Lemma {\ref{Lem3.15}}. Therefore $(D_{12} + \Dp_{12})$ defines a birational morphism $\Xt \rightarrow \Vt$ where $\Vt$ is a smooth Del Pezzo surface of degree $4$. In particular, this $\Vt$ is a minimal resolution of an Iskovskih surface.

	From the above, $\Xt$ is not minimal in any cases.
\end{proof}
\begin{rem}
	The configurations of $\Xt$ of Proposition {\ref{Prop3.18}} are the following, where vertices $\circ$ and $\bullet$ are $(-2)$curves on $\Xt$ and two $(-1)$curves in proof of Proposition {\ref{Prop3.18}}, respectively. Each edge represents the intersection of two vertices.

$$(1)
\xygraph{
	\bullet (
		- []!{+(-1,-.5)} \circ,
		- []!{+(0,-.5)} \circ,
		- []!{+(1,-.5)} \circ
		 [r] \circ,
		)
}
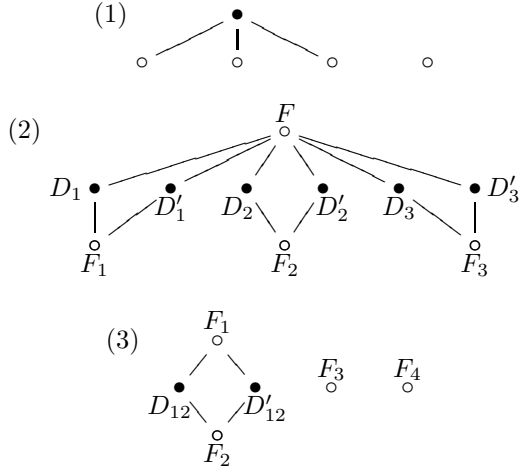
$$
$$
(2)
\xygraph{
	\circ ([]!{+(0,+.2)} {F}) (
		- []!{+(-2,-.6)} \bullet ([]!{+(-.3,0)} {D_1})
		- []!{+(0,-.6)} \circ ([]!{+(0,-.2)} {F_1}),
		- []!{+(-1.2,-.6)} \bullet ([]!{+(0,-.2)} {\Dp_1})
		- []!{+(-.8,-.6)} \circ,
		- []!{+(-.4,-.6)} \bullet ([]!{+(-.1,-.2)} {D_2})
		- []!{+(.4,-.6)} \circ ([]!{+(0,-.2)} {F_2}),
		- []!{+(.4,-.6)} \bullet ([]!{+(.1,-.2)} {\Dp_2})
		- []!{+(-.4,-.6)} \circ,
		- []!{+(2,-.6)} \bullet ([]!{+(.3,0)} {\Dp_3})
		- []!{+(0,-.6)} \circ ([]!{+(0,-.2)} {F_3}),
		- []!{+(1.2,-.6)} \bullet ([]!{+(0,-.2)} {D_3})
		- []!{+(.8,-.6)} \circ,
		)
}
$$
$$
(3)
\xygraph{
	\circ ([]!{+(0,+.2)} {F_1}) (
		- []!{+(-.4,-.5)} \bullet ([]!{+(-.1,-.2)} {D_{12}})
		- []!{+(.4,-.5)} \circ ([]!{+(0,-.2)} {F_2}),
		- []!{+(.4,-.5)} \bullet ([]!{+(.1,-.2)} {\Dp_{12}})(
			[]!{+(.8,0)} \circ ([]!{+(0,+.2)} {F_3})
			[]!{+(.8,0)} \circ ([]!{+(0,+.2)} {F_4}),
			- []!{+(-.4,-.5)} \circ
			)
		)
}
$$
Thus, $\Vt$ of the proof of Proposition {\ref{Prop3.18}} is:
\begin{itemize}
	\item[] the minimal resolution of a RDP Del Pezzo surface of degree $3$ with $A_1$ singularity if (1);
	\item[] a smooth Del Pezzo surface of degree $8$ if (2);
	\item[] an Iskovskih surface if (3)
\end{itemize}

Moreover, by Proposition {\ref{Prop:bir degree3}} and {\ref{Prop:bir degree4}} $\Xt$ is
\begin{itemize}
	\item[] $k$-rational if (1);
	\item[] $k$-rational if (2) and $\Xt$ has a $k$-point;
	\item[] $k$-unirational if (3) and $\Xt$ has a $k$-point.
\end{itemize}
\end{rem}

\begin{thm}{\label{Thm:FourSingularities}}
	Let $X$ be a RDP Del Pezzo surface of degree $2$ over a perfect field $k$. Suppose that $X \times_k \kb$ has just four singular points. If not $\ctext{3}$ of Theorem {\ref{MainThm1}}, then $\Xt$ is not minimal.
\end{thm}
\begin{proof}
	This follows from Proposition {\ref{Prop3.17}} and {\ref{Prop3.18}}.
\end{proof}
\subsection{The case with at least five singularities}
	Let $\delta$ be the number of singularities of $\Xkb$. Suppose $\delta \geq 5$. Then $\Xkb$ has $\delta A_1$ singularities by Proposition {\ref{Prop:singularity type}}. Thus the following lemma holds.

\begin{lem}{\label{Lem3.19}}
	 Suppose that $\delta \geq 5$. For three $(-2)$curves $F_1, F_2, F_3$ on $\Xt$, there exists a $(-1)$curve $E$ satisfying the followings:
\begin{itemize}
	\item[(1)] $E$ intersects at least two of $\{ F_1, F_2, F_3 \}$.
	\item[(2)] $E$ is a $(-1)$curve of Lemma {\ref{Lem3.15}}.
\end{itemize} 
\end{lem}
\begin{proof}
	This follows from Lemma {\ref{Lem3.15}} and {\ref{Lem3.16}}.
\end{proof}

\begin{lem}{\label{Lem3.21}}
	Suppose that $\delta \geq 5$. Let $E$ and $\Ep$ be two $(-1)$curves of Lemma {\ref{Lem3.15}} and let $F_i$ and $\Fp_i$ be $(-2)$curves such that $F_i \cdot E = 1$ and $\Fp_i \cdot \Ep = 1$ $(i = 1,2,3)$. Then $\# (\{ F_1, F_2, F_3 \} \cap \{ \Fp_1, \Fp_2, \Fp_3 \} )= 1$ and $E \cdot \Ep = 0$.
\end{lem}
\begin{proof}
	Let $n := \# (\{ F_1, F_2, F_3 \} \cap \{ \Fp_1, \Fp_2, \Fp_3 \} )$. By Lemma {\ref{Lem:(-1)curveofLem}}, we have $2E \in |-\Kxt - F_1 - F_2 - F_3|$ and $2\Ep \in |-\Kxt - \Fp_1 - \Fp_2 - \Fp_3|$. Therefore $2E \cdot 2\Ep = 2 - 2n \geq 0$ and thus $n \leq 1$. However, $n = 0$ is impossible since $E \cdot \Ep$ is not an integer when $n = 0$. Therefore, $n = 1$ and $E \cdot \Ep = 0$. 
\end{proof}

	In the proofs of Proposition {\ref{Prop3.20}}, {\ref{Prop3.21}} and {\ref{Prop3.22}}, we use only Lemma {\ref{Lem3.19}} and {\ref{Lem3.21}}.
\begin{prop}{\label{Prop3.20}}
	If $\delta = 5$, $X$ is birationally equivalent to a smooth Del Pezzo surface of degree $4$. In particular, $X$ is $k$-unirational if $\Xt$ has a $k$-point.
\end{prop}
\begin{proof}
	By using Lemma {\ref{Lem3.19}} twice, we obtain two $(-1)$curves of Lemma {\ref{Lem3.15}}, denoted by $E_1$ and $E_2$. We show there is no other $(-1)$curves of Lemma {\ref{Lem3.15}}. Let $F_1, F_2, F_3$ be $(-2)$curves on $\Xt$. $E_1$ or $E_2$ intersects with at least two of $\{ F_1, F_2, F_3 \}$. Thus it follows from Lemma {\ref{Lem3.21}} that, except $E_1, E_2$, there is no $(-1)$curve of Lemma {\ref{Lem3.15}} which intersects $F_1, F_2, F_3$. Therefore $(E_1 + E_2)$ defines a birational morphism $\Xt \rightarrow \Vt$ where $\Vt$ is a smooth Del Pezzo surface of degree $4$.
\end{proof}

\begin{prop}{\label{Prop3.21}}
	If $\delta = 6$, $X$ is birationally equivalent to a smooth Del Pezzo surface of degree $6$. In particular, $X$ is $k$-rational if $\Xt$ has a $k$-point.
\end{prop}
\begin{proof}
	Let $F$ be a $(-2)$curve on $\Xt$ and $n$ be the number of $(-1)$curves of Lemma {\ref{Lem3.15}} which intersect $F$. We show that $n = 2$. 

	Suppose that $n \geq 3$. Let $E_i$ be the $n$ $(-1)$curves of Lemma {\ref{Lem3.15}} $(i = 1,\cdots,n)$. By Lemma {\ref{Lem3.21}}, each $E_i$ intersects two $(-2)$curves except $F$ and the $2n$ $(-2)$curves are distinct. This contradicts $\delta = 6$.

	Suppose that $n = 0$. Similarly as Proposition {\ref{Prop3.20}}, the five $(-2)$curves except $F$ define two $(-1)$curves of Lemma {\ref{Lem3.15}}, denoted by $E_1, E_2$. Let $F_1$ be a $(-2)$curve such that $E_1 \cdot F_1 = 1$ and $E_2 \cdot F_1 = 0$. Let $F_2$ be a $(-2)$curve such that $E_1 \cdot F_2 = 0$ and $E_2 \cdot F_2 = 1$. By using Lemma {\ref{Lem3.19}} for $F$, $F_1$ and $F_2$, we obtain a new $(-1)$curve of Lemma {\ref{Lem3.15}}. This is impossible since there exist only two $(-1)$curves of Lemma {\ref{Lem3.15}} in the case of $\delta = 5$.

	Suppose that $n = 1$. Let $E$ be the $(-1)$curve of Lemma {\ref{Lem3.15}} intersecting $F$ and let $F_1, F_2$ be the $(-2)$curves intersecting $E$. Let $F_3, F_4, F_5$ be the other $(-2)$curves. By using Lemma {\ref{Lem3.19}} for $F$, $F_i$ and $F_j$ $(3 \leq i,j \leq 5)$, there exists a $(-1)$curve of Lemma {\ref{Lem3.15}} denoted by $E_{ij}$. Since $E_{ij} \cdot F = 0$, $E_{ij}$ intersects $F_i$, $F_j$ and one of $\{ F_1, F_2 \}$ by Lemma {\ref{Lem3.21}}. Thus one of $\{ F_1, F_2\}$ intersects two of $\{ E_{34}, E_{35}, E_{45}\}$. Then, $F_1$ or $F_2$ intersects three $(-1)$curves of Lemma {\ref{Lem3.15}}. This is impossible just as $n = 3$ is impossible.

	From the above, each $(-2)$curve intersects two $(-1)$curves of Lemma {\ref{Lem3.15}}. Thus $\Xt$ has $2 \times 6 / 3 = 4$ $(-1)$curves of Lemma {\ref{Lem3.15}}. By collapsing this four $(-1)$curves, $\Xt$ is birationally equivalent to a smooth Del Pezzo surface of degree $6$.
\end{proof}

\begin{prop}{\label{Prop3.22}}
	If $\delta = 7$, $X$ is birationally equivalent to a smooth Del Pezzo surface of degree $9$. In particular, $X$ is $k$-rational if $\Xt$ has a $k$-point.
\end{prop}
\begin{proof}
	Let $F$ be a $(-2)$curve on $\Xt$ and let $n$ be the number of $(-1)$curves of Lemma {\ref{Lem3.15}} intersecting $F$. We show that $n =  3$.

	Suppose that $n \geq 4$. Let $E_i$ be the $n$ $(-1)$curves of Lemma {\ref{Lem3.15}} intersecting $F$ $( i = 1,\cdots,n)$. Each of $E_i$ intersects two $(-2)$curves except $F$. This contradicts $\delta = 7$.

	Similarly as the case of $\delta = 6$, we have $n \geq 2$. Suppose that $n = 2$. Let $E_1$ and $E_2$ be the two $(-1)$curves of Lemma {\ref{Lem3.15}} intersecting $F$. Let $F_1$ and $F_2$ be the remaining two $(-2)$curves intersecting $E_1$ and let $F_3$ and $F_4$ be the remaining two $(-2)$curves intersecting $E_2$. Let $F_5$ and $F_6$ be the other $(-2)$curves. By using Lemma {\ref{Lem3.19}} for $F$, $F_5$ and $F_6$, there exists a $(-1)$curve of Lemma {\ref{Lem3.15}} intersecting $F_5$ and $F_6$, denoted by $E_3$. By Lemma {\ref{Lem3.21}}, $E_3$ intersects one of $\{ F_1, F_2\}$ and one of $\{ F_3, F_4\}$. This is impossible since a $(-1)$curve of Lemma {\ref{Lem3.15}} does not intersect four $(-2)$curves. 

	From the above, each $(-2)$curve intersects three $(-1)$curves of Lemma {\ref{Lem3.15}}. Thus $\Xt$ has $3 \times 7 / 3 = 7$ $(-1)$curves of Lemma {\ref{Lem3.15}} for all $(-2)$curves. By collapsing this seven $(-1)$curves, $\Xt$ is birationally equivalent to a smooth Del Pezzo surface of degree $9$.
\end{proof}
\subsection{The case with one singularity}
	If $X \times \kb$ has only one singularity, the type of the singularity is either $A_1$, $A_2$, $A_3$, $A_4$, $A_5$, $A_6$, $A_7$, $D_4$, $D_5$, $D_6$, $E_6$ or $E_7$ by Proposition {\ref{Prop:singularity type}}. In this subsection, we show the following theorem:
\begin{thm}{\label{MainThm3.5}}
	Let $X$ be a RDP Del Pezzo surface of degree $2$ with only one singularity. If neither
\begin{itemize}
	\item[$\ctext{1}$] the singularity of $X$ is type $A_1$ nor
	\item[$\ctext{2}$] the singularity of $X$ is type $A_2$ and the two $(-2)$curves on $\Xt$ are conjugate
\end{itemize}
then $\Xt$ is not minimal.
\end{thm}

	At first, we show this in case of $A_2$ type.
\begin{prop}{\label{Prop:A2case}}
	If the singularity of $\Xkb$ is type $A_2$ and the two $(-2)$curves on $\Xt$ are not conjugate, then each $(-2)$curve intersects six $(-1)$curves and the six $(-1)$curves do not intersect each other. In particular, $\Xt$ is birationally equivalent to a weak Del Pezzo surface of degree $8$ with a $(-2)$curve.
\end{prop}
\begin{proof}
	Let $F_1$ and $F_2$ be the two $(-2)$curves defined over $k$.
$$
\xygraph{
	\circ ([]!{+(0,+.2)} {F_1}) - [r]
	\circ ([]!{+(0,+.2)} {F_2})
}
$$
By Proposition {\ref{Prop3.7}}, there exist six pairs of two pre$(-1)$curves $\{ (D_i, \Dp_i)\}_{i = 1,\cdots, 6}$ such that $D_i \cdot F_1 = \Dp_i \cdot F_1 = 1$ and $D_i \cdot \Dp_i = 1$ for all $i$. Since $(D_i + \Dp_i) \in |-\Kxt - F_1|$, we have $(D_i + \Dp_i) \cdot F_2 = -1$. Thus, we can assume that $D_i \cdot F_2 = 0$ and $\Dp_i \cdot F_2 = -1$ for $i = 1, \cdots, 6$. Then $D_i$ are $(-1)$curves by Proposition {\ref{Prop3.4}}. Moreover, $F_1$ is defined over $k$ and $D_1, D_2, \cdots, D_6$ are the only $(-1)$curves intersecting $F_1$. Therefore, $(D_1 + D_2 + \cdots +D_6)$ is also defined over $k$. These six $(-1)$curves do not intersect with each other. Let $\pi : \Xt \rightarrow \Vt$ be the blowing down defined by $(D_1 + D_2 + \cdots + D_6)$. Then $\Vt$ is a weak Del Pezzo surface of degree $8$ with a $(-2)$curve $\pi(F_2)$.
\end{proof}

	Next, we suppose that $\Xt$ has at least three $(-2)$curves. Then there exist two $(-2)$curves $F, \Fp$ which do not intersect. Thus, we obtain two pre$(-1)$curves which intersect $F$ and $\Fp$ by Proposition {\ref{Prop3.8}}. However, we did not prove that the two $(-1)$curves contained in the pre$(-1)$curves do not intersect (cf. Proposition {\ref{Prop3.9}}). We will show this in Proposition {\ref{Prop:MainProp3.5}}. By collapsing these two $(-1)$curves, almost all part of Theorem {\ref{MainThm3.5}} is proved.

	We show the following two lemmas for Proposition {\ref{Prop:MainProp3.5}}.

\begin{lem}{\label{Lem:no(-1)curve}}
	Suppose that $X \times \kb$ has only one singularity. Let $F, \Fp$ be $(-2)$curves such that $F \cdot \Fp = 0$. Then there exist $(-2)$curve $F_1, F_2, \cdots, F_n$ such that
$$
\xygraph{
	\circ ([]!{+(0,+.2)} {F = F_1}) - [r]
	\circ ([]!{+(0,+.2)} {F_2}) - [r]
	\cdots - [r]
	\circ ([]!{+(0,+.2)} {F_{n - 1}}) - [r]
	\circ ([]!{+(0,+.2)} {\Fp = F_n})
}.
$$
Let $D$ be a pre$(-1)$curve such that $D \cdot F = D \cdot \Fp = 1$. Then,
\begin{itemize}
	\item[(1)] there exists $i$ such that $D \cdot F_i = -1$,
	\item[(2)] the following are equivalent:
	\begin{itemize}
		\item[(i)] there uniquely exists $i$ such that $D \cdot F_i = -1$,
		\item[(ii)] $D \cdot F_j \neq 1$ for all $j = 2, \cdots, n-1$,
	\end{itemize}
	\item[(3)] $D \cdot (F_1 + \cdots + F_n) = 1$,
	\item[(4)] $(D - F_2 - \cdots - F_{n-1})$ is a pre$(-1)$curve. In particular, $F_2, \cdots, F_{n-1}$ are components of $D$.
\end{itemize}
\end{lem}
\begin{proof}
	We show (1). Suppose that $D \cdot F_i \geq 0$ for all $i$. Since $D \cdot F_1 = 1$, we have $(D + F_1)$ is a pre$(-1)$curve by Lemma {\ref{Lem:minusF}}. Further, since $1 \geq (D + F_1) \cdot F_2 = D \cdot F_2 + 1 \geq 1$ by Lemma {\ref{Lem:intersec_No}}, we have $(D + F_1) \cdot F_2 = 1$. Thus $(D + F_1 + F_2)$ is also a pre$(-1)$curve. By repeating this, $(D + F_1 + F_2 + \cdots + F_{n-1})$ is a pre$(-1)$curve. However, $(D + F_1 + F_2 + \cdots + F_{n-1}) \cdot F_n = 2$. This is impossible by Lemma {\ref{Lem:intersec_No}}. Thus there exists $i$ such that $D \cdot F_i = -1$.

	Next we show (2). ((i)$\Rightarrow$(ii)) Suppose that there exists $j \in \{ 2, \cdots, n-1\}$ such that $D \cdot F_j = 1$. Then, since $1 \geq (D + F_1) \cdot F_j = 1 + F_1 \cdot F_j \geq 1$ by Lemma {\ref{Lem:intersec_No}}, we have $F_1 \cdot F_j = 0$ and thus $j \neq 2$. Similarly, we have $j \neq n-1$. Thus, there exists $i_1 \in \{2, \cdots, j - 1\}$ and $i_2 \in \{ j + 1, \cdots, n - 1\}$ such that $D \cdot F_{i_1} = D \cdot F_{i_2} = -1$ by (1). Therefore $i$ is not unique. 

	((ii)$\Rightarrow$(i)) By (1), there exists a $(-2)$curve $F_i$ such that $D \cdot F_i = -1$. Suppose that $D \cdot F_j \neq 1$ for all $j = 2, \cdots, n-1$. Since $D \cdot F_i = -1$, we have $(D - F_i)$ is also a pre$(-1)$curve by Lemma {\ref{Lem:minusF}}. Further, since $-1 \leq (D - F_i) \cdot F_{i + 1} = D \cdot F_{i + 1} - 1 \leq -1$ by Lemma {\ref{Lem:intersec_No}}, we have $D \cdot F_{i + 1} = 0$ and $(D - F_i) \cdot F_{i + 1} = -1$. Then $(D - F_i - F_{i + 1})$ is also a pre$(-1)$curve. By repeating this, $D \cdot F_j = 0$ for $j = i + 1, \cdots, n-1$ and $(D - F_i - F_{i + 1} - \cdots -F_{n - 1})$ is a pre$(-1)$curve. Since $(D - F_i - F_{i + 1} - \cdots -F_{n - 1}) \cdot F_{i - 1} = D \cdot F_{i - 1} - 1$, we have $D \cdot F_{i - 1} = 0$ and $(D - F_{i-1} -F_i - \cdots - F_{n-1})$ is also a pre$(-1)$curve. By repeating this, we have $D \cdot F_j = 0$ for $j \neq 1, i, n$ and $(D - F_2 - \cdots - F_{n - 1})$ is a pre$(-1)$curve (this is a part of (4)).

	Let $k := \# \{ i = 1,\cdots, n | D \cdot F_i = 1\}$ $(k \geq 2)$. We show (3) by the induction on $k$. If $k = 2$, (3) is clear by (2). If $k \geq 2$, let $\{ a_1, a_2, \cdots, a_k\}$ be a subsequence of $\{1, \cdots, n\}$ such that $D \cdot F_{a_1} = \cdots = D \cdot F_{a_k} = 1$ $(1 = a_1 < a_2 < \cdots < a_k = n)$. By the induction hypothesis, $D \cdot (F_1 + \cdots + F_{a_{k-1}}) = 1$. Further, $D \cdot (F_{a_{k-1}} + \cdots + F_n) = 1$. Thus $D \cdot (F_1 + \cdots + F_n) = 1$ for any $k$.

	Finally, we show (4). $D$ satisfies $(D - F_2 - \cdots - F_{n-1}) \cdot \Kxt = -1$ and
$$(D - F_2 - \cdots - F_{n-1})^2 = (D)^2 + (F_2 + \cdots + F_{n-1})^2 - 2D \cdot(F_2 + \cdots + F_{n-1}) = -1.$$
Therefore, it suffices to show that $(D - F_2 - \cdots - F_{n-1})$ is effective. We show this by induction on $k$. If $k = 2$, this is already proved in the proof of (2). If $k > 2$, then $(D - F_{a_2 + 1} - \cdots - F_{n-1})$ is a pre$(-1)$curve by the induction hypothesis. Then
\begin{eqnarray}
	& &(D - F_{a_2 + 1} - \cdots - F_{n-1}) \cdot (F_{a_2 + 1} + \cdots + F_{n-1}) \nonumber \\
	& &= D \cdot (F_{a_2 + 1} + \cdots + F_{n-1}) - (F_{a_2 + 1} + \cdots + F_{n-1})^2 = 1 \nonumber
\end{eqnarray}
Thus there exists $l \in \{ a_2 + 1, \cdots, n-1\}$ such that $(D - F_{a_2 + 1} - \cdots - F_{n-1}) \cdot F_l = 1$. Let $m$ be the minimal $l$. Then, $(D - F_{a_2 + 1} - \cdots - F_{n-1}) \cdot F_1 = (D - F_{a_2 + 1} - \cdots - F_{n-1}) \cdot F_m = 1$ and $(D - F_{a_2 + 1} - \cdots - F_{n-1}) \cdot F_j \neq 1$ for all $j = 2, \cdots, m-1$. Therefore, $F_2, \cdots, F_{m-1}$ are contained in $(D - F_{a_2 + 1} - \cdots - F_{n-1})$ by the induction hypothesis. Thus $D$ contains $F_2, \cdots, F_{n-1}$ and $(D - F_2 - \cdots - F_{n-1})$ is a pre$(-1)$curve.
\end{proof}

\begin{lem}{\label{Lem:NoIntersection}}
	Suppose that $\delta = 1$. Let $\mF$ be the sum of all $(-2)$curves on $\Xt$ and $F, \Fp$ be two $(-2)$curves. Let $D$ be a pre$(-1)$curve such that $D \cdot F = D \cdot \Fp = 1$. If each of $F$ and $\Fp$ intersects only one $(-2)$curve (that is, $F$ and $\Fp$ are ``terminal'' on $\mF$), then $(D + F + \Fp - \mF)$ is a pre$(-1)$curve. In particular, $\mF$ is contained in $(D + F + \Fp)$ as a component.
\end{lem}
\begin{proof}
	Let $F_1, F_2, \cdots, F_n$ be $(-2)$curves such that:
$$
\xygraph{
	\circ ([]!{+(0,-.2)} {F = F_1}) - [r]
	\circ ([]!{+(0,-.2)} {F_2}) - [r]
	\cdots - [r]
	\circ ([]!{+(0,-.2)} {F_{n - 1}}) - [r]
	\circ ([]!{+(0,-.2)} {\Fp = F_n})
}
$$
Let $\mF_t := F_2 + F_3 + \cdots + F_{n-1}$ and $\mF_b := \mF - \mF_t - F_1 - F_n$. Let $G_1, G_2, \cdots, G_m$ be the $(-2)$curves in $\mF_b$ such that:
$$
\xygraph{
	\circ ([]!{+(0,-.2)} {F = F_1}) - [r]
	\circ ([]!{+(0,-.2)} {F_2}) - [r]
	\cdots - [r]
	\circ (
		- [r] \cdots
		- [r] \circ ([]!{+(0,-.2)} {\Fp = F_n})
		 [r]
		 [r] []!{+(-.3,0)} {\biggr\} \mF_t + F + \Fp},
		- [u] \circ ([]!{+(0,+.2)} {G_1})
		- [r] \circ ([]!{+(0,+.2)} {G_2})
		- [r] \cdots
		- [r] \circ ([]!{+(0,+.2)} {G_m})
		 [r] []!{+(-.3,0)} {\biggr\} \mF_b}
		)
}
$$

	We show that $D \cdot G_l = 0$ and $(D - \mF_t - G_1 - \cdots - G_l)$ is a pre$(-1)$curve for $l = 1, \cdots, m$ by induction on $l$.

	If $l = 0$, this is Lemma {\ref{Lem:no(-1)curve}} (4). If $l > 0$, $(D - \mF_t - G_1 - \cdots - G_{l - 1})$ is a pre$(-1)$curve by the induction hypothesis. Then $(D - \mF_t - G_1 - \cdots - G_{l - 1}) \cdot G_l = D \cdot G_l - 1$. Thus $D \cdot G_l = 0$ or $1$. We show $D \cdot G_l = 0$. If $D \cdot G_l = 1$, then $(D + G_l)$ is a pre$(-1)$curve by Lemma {\ref{Lem:minusF}}. Further, $(D + G_l) \cdot F = (D + G_l) \cdot \Fp = 1$. Thus we have $(D + (D + G_l)) = (2D + G_l) \in |-\Kxt - F - \Fp|$ by Proposition {\ref{Prop3.8}}. However, this is impossible since
	$$(-\Kxt - F - \Fp - G_l) \cdot (\mF_t + G_1 + G_2 + \cdots + G_{l-1}) = -3.$$
Thus $D \cdot G_l = 0$. Then, $(D - \mF_t - G_1 - \cdots - G_{l - 1}) \cdot G_l = -1$. By Lemma {\ref{Lem:minusF}}, we have $(D - \mF_t - G_1 - \cdots - G_{l - 1} - G_l)$ is a pre$(-1)$curve for all $l$. In particular, $(D - \mF_t - \mF_b)$ is a pre$(-1)$curve.
\end{proof}

\begin{prop}{\label{Prop:MainProp3.5}}
	Suppose that $\delta = 1$ and $\Xt$ has at least three $(-2)$curves. Let $F, \Fp$ be two $(-2)$curves which intersect only one $(-2)$curve. Let $D_1, D_2$ be the two pre$(-1)$curves which intersect $F$ and $\Fp$ and let $E_1$ and $E_2$ be the $(-1)$curves contained in $D_1$ and $D_2$, respectively. Then, $E_1 \cdot E_2 = 0$ or $-1$.
\end{prop}
\begin{proof}
	Let $\Dp_i := D_i - E_i$ for $i = 1,2$. Since $(D_1 + D_2) \in |-\Kxt - F - \Fp|$, we have $(E_1 + E_2) \in |-\Kxt - F - \Fp - \Dp_1 - \Dp_2|$. Thus,
\begin{eqnarray}
	-2 + 2E_1 \cdot E_2 &=& (E_1 + E_2)^2 \nonumber \\ 
	 &=& (-\Kxt - F - \Fp - \Dp_1 - \Dp_2) \cdot (E_1 + E_2) \nonumber \\
	&=& 2 - (F + \Fp + \Dp_1 + \Dp_2) \cdot (E_1 + E_2) {\label{eq:3.4.1}}
\end{eqnarray}
Since $-1 = (D_1)^2 = (E_1 + \Dp_1)^2 = -1 + 2E_1 \cdot \Dp_1 + (\Dp_1)^2$ and $(\Dp_1)^2 \leq -2$, we have $\Dp_1 \cdot E_1 \geq 1$. On the other hand, $(F + \Fp + \Dp_2) \cdot E_1 \geq 1$ since $(F + \Fp + \Dp_2)$ contains all $(-2)$curves by Lemma {\ref{Lem:NoIntersection}} and $E_1$ intersects at least one $(-2)$curve by Lemma {\ref{Lem3.12}}. Thus, $(F + \Fp + \Dp_1 + \Dp_2) \cdot E_1 \geq 2$. Similarly, $(F + \Fp + \Dp_1 + \Dp_2) \cdot E_2 \geq 2$. Thus $E_1 \cdot E_2 \leq 0$ by the equation ({\ref{eq:3.4.1}}) 
\end{proof}
\begin{proof}[Proof of Theorem {\ref{MainThm3.5}}]
	If the type of singularity of $\Xkb$ is $A_2$, this follows from Proposition {\ref{Prop:A2case}}. Suppose that the type of singularity of $\Xkb$ is not $A_1$ or $A_2$. Then the type is either $A_3$, $A_4$, $A_5$, $A_6$, $A_7$, $D_4$, $D_5$, $D_6$, $E_6$ or $E_7$ by Proposition {\ref{Prop:singularity type}}. Except $D_4$ type, there exist $(-2)$curves $F, \Fp$ which satisfy both of the following:
\begin{itemize}
	\item[] $(F + \Fp)$ is defined over $k$;
	\item[] each $F$ and $\Fp$ intersects only one $(-2)$curves.
\end{itemize}
For example, if the type is $A_3$, $D_5$ or $E_6$ then $F$ and $\Fp$ are:
$$
\xygraph{
	{A_1}
	[]!{+(.1,-.6)} \circ ([]!{+(0,+.2)} {F})
	- []!{+(.8,0)} \circ 
	- []!{+(.8,0)} \circ ([]!{+(0,+.2)} {\Fp})
}
\xygraph{
	{D_5}
	[]!{+(.1,-.6)} \circ ([]!{+(0,+.2)} {F})
	- []!{+(.8,0)} \circ (
		- []!{+(0,.6)} \circ ([]!{+(+.2,0)} {\Fp}),
		- []!{+(.8,0)} \circ
			- []!{+(.8,0)} \circ
	)
}
\xygraph{
	{E_6}
	[]!{+(.1,-.6)} \circ ([]!{+(0,+.2)} {F})
	- []!{+(.8,0)} \circ
	- []!{+(.8,0)} \circ (
		- []!{+(0,.6)} \circ,
		- []!{+(.8,0)} \circ
			- []!{+(.8,0)} \circ ([]!{+(0,+.2)} {\Fp})
	)
}.
$$
By using Proposition {\ref{Prop:MainProp3.5}} for the $F, \Fp$, we obtain one or two $(-1)$curves which do not intersect and the sum of these is defined over $k$. By collapsing these $(-1)$curves, $\Xt$ is birationally equivalent to a weak Del Pezzo surface of degree $3$ or $4$. This is $k$-unirational by Proposition {\ref{Prop:bir degree3}} and {\ref{Prop:bir degree4}}

	Suppose that the type of singularity of $\Xkb$ is $D_4$. Let $F_1, F_2, F_3$ and $F_4$ be the four $(-2)$curves such that:
$$
\xygraph{
	\circ ([]!{+(0,-.2)} {F_1}) - []!{+(.7,0)}
	\circ ([]!{+(0,-.2)} {F_3}) (
		- []!{+(.7,0)} \circ ([]!{+(0,-.2)} {F_4}),
		- []!{+(0,.5)} \circ ([]!{+(.3,0)} {F_2})
		)
}
$$
Let $D_{12}, \Dp_{12}$ be pre$(-1)$curves which intersect $F_1$ and $F_2$ and let $E_{12}$ and $\Ep_{12}$ be the $(-1)$curves contained in $D_{12}$ and $\Dp_{12}$, respectively. Similarly, $(F_1, F_4)$ defines $D_{14}, \Dp_{14}, E_{14}$ and $\Ep_{14}$ and $(F_2, F_4)$ defines $D_{24}, \Dp_{24}, E_{24}$ and $\Ep_{24}$. Then $(E_{12} + \Ep_{12} + E_{14} + \Ep_{14} + E_{24} + \Ep_{24})$ is defined over $k$. We show that the six $(-1)$curves do not intersect each other. By Lemma {\ref{Lem:NoIntersection}}, we have $(D_{12} - F_3 - F_4)$ and $(\Dp_{12} - F_3 - F_4)$ are pre$(-1)$curves. Since $D_{12} \cdot F_3 = \Dp_{12} \cdot F_3 = -1$, $D_{12} \cdot F_4 = \Dp_{12} \cdot F_4 = 0$ and $D_{12} \cdot \Dp_{12} = 0$, we have
\begin{itemize}
	\item[] $(D_{12} - F_3 - F_4) \cdot (\Dp_{12} - F_3 - F_4) = 0$,
	\item[] $(D_{12} - F_3 - F_4) \cdot F_1 = (\Dp_{12} - F_3 - F_4) \cdot F_1 = 0$,
	\item[] $(D_{12} - F_3 - F_4) \cdot F_2 = (\Dp_{12} - F_3 - F_4) \cdot F_2 = 0$,
	\item[] $(D_{12} - F_3 - F_4) \cdot F_3 = (\Dp_{12} - F_3 - F_4) \cdot F_3 = 0$ and
	\item[] $(D_{12} - F_3 - F_4) \cdot F_4 = (\Dp_{12} - F_3 - F_4) \cdot F_4 = 1$.
\end{itemize}
Thus $E_{12} = D_{12} - F_3 - F_4$, $\Ep_{12} = \Dp_{12} - F_3 - F_4$ and $E_{12}$ does not intersect $\Ep_{12}$. Similarly, $E_{14} = D_{14} - F_3 - F_2$, $\Ep_{14} = \Dp_{14} - F_3 - F_2$, $E_{24} = D_{24} - F_3 - F_1$ and $\Ep_{24} = \Dp_{24} - F_3 - F_1$. We show $E_{12} \cdot E_{14} = 0$. $E_{12} \cdot E_{14} = (D_{12} - F_3 - F_4) \cdot (D_{14} - F_2 - F_3) = D_{12} \cdot D_{14}$. If not $D_{12} \cdot D_{14} = 0$, we have $D_{12} \cdot D_{14} = -1$ or $1$ by Proposition {\ref{Prop3.7}}. Since $D_{12} \neq D_{14}$, we have $D_{12} \cdot D_{14} \neq -1$. If $D_{12} \cdot D_{14} = 1$, then $(D_{12} + D_{14}) \in |-\Kxt - F_1|$ by Proposition {\ref{Prop3.7}}. However, $(D_{12} + D_{14}) \cdot F_2 = 1$ but $(-\Kxt - F_1) \cdot F_2 = 0$. Therefore $D_{12} \cdot D_{14} = 0$ and $E_{12} \cdot E_{14} = 0$. Similarly the six $(-1)$curves $E_{12}, \Ep_{12}, E_{14}, \Ep_{14}, E_{24}, \Ep_{24}$ do not intersect each other. By collapsing $(E_{12} + \Ep_{12} + E_{14} + \Ep_{14} + E_{24} + \Ep_{24})$, $\Xt$ is birationally equivalent to a weak Del Pezzo surface of degree $8$. This is $k$-rational by Proposition {\ref{Prop:bir degree>5}}
\end{proof}

	As in Section 3.2, 3.3, the unirationality of $X$ is classified associated to the type of the singularity of $X$. See Appendix A.

\begin{proof}[Proof of Theorem {\ref{MainThm1}}]
	This follows from Theorem {\ref{Thm3.11}}, {\ref{Thm3.14}}, {\ref{Thm:FourSingularities}}, Proposition {\ref{Prop3.20}}, {\ref{Prop3.21}}, {\ref{Prop3.22}}, Theorem {\ref{MainThm3.5}}.
\end{proof}

\section{Uniraionality over arbitrary fields}
	In this section, we show Theorem {\ref{MainThm2}}. Let $X$ be a RDP Del Pezzo surface of degree $2$ over a field $k$. Then, $X$ is a quartic surface of weighted projective space $\bP_k(1,1,1,2)$. Since the point $p = (0, 0, 0, 1) \in \bP_k(1,1,1,2)$ is a singular point which is not rational double point, $X$ does not pass through $p$. Thus the projection from $p$ defines a finite morphism $\kappa : X \rightarrow \bP_k^2$ of degree $2$. Note that $\kappa$ is defined by $|-\Kx|$. 

	In Section 4.1, we show (1) of Theorem {\ref{MainThm2}}. We assume that $\kappa$ is separable in the remaining subsections. In Section 4.2, we define ``spine'' of a weak Del Pezzo surface at a point. In Section 4.3, we consider the $(-2)$curves on the weak Del Pezzo surface of degree $1$ defined by blowing up of $\Xt$. Then, we show that spines can become $(-2)$curves. In Section 4.4, we show Theorem {\ref{MainThm2}}.
\subsection{Inseparable anti-canonical morphism}
	Suppose that the anti-canonical morphism $\kappa$ is purely inseparable. Then the characteristic of $k$ is $2$ and $X$ is defined by an equation $w^2 + q_4(x,y,z) = 0$ in $\bP(x,y,z,w)$ where $\deg(x) = \deg(y) = \deg(z) = 1, \deg(w) = 2$ and $\deg(q_4) = 4$. 

	I would like to thank T. Kawakami for showing me the following proposition.
\begin{prop}{\label{Prop:Kawakami}}
	Let $f:X \rightarrow Y$ be a dominant, finite and purely inseparable morphism of normal varieties. Then there exists a morphism $g : Y \rightarrow X$ such that $f \circ g$ is finite succession of the Frobenius morphisms of $Y$.
$$
\xymatrix{
	Y \ar[rr]^g \ar[rd]^{Fr} & & X \ar[rr]^f & & Y\\
	 & Y \ar[r]^{Fr} & \cdots \ar[r]^{Fr}& Y \ar[ru]^{Fr}\\
}
$$
\end{prop}
\begin{proof}
	Let $\Spec A$ be an affine open set of $Y$ and $\Spec B := f^{-1}(\Spec A) \subset X$. Then $f$ induces an inclusion $A \subset B$ and an extension $K(A) \subset K(B)$. Since $f$ is purely inseparable, there exists $n$ such that $b \in K(B) \Rightarrow b^{p^n} \in K(A)$. On the other hand, $K(A) \cap B = A$ since $A$ is a normal ring and $B$ is integral over $A$. Thus there exists a homomorphism $B \rightarrow A ; b \mapsto b^{p^n}$. This induces the morphism $g|_{\Spec A} : \Spec A \rightarrow \Spec B$ such that $f \circ g|_{\Spec A}$ is the composite of $n$ Frobenius morphisms. By gluing $g|_{\Spec A}$ on an affine covering, we obtain the desired $g$.
\end{proof}
	Note that this $g$ is not defined over $k$ since the Frobenius morphism is not defined over $k$.

\begin{cor}{\label{Cor:MainThm2(1)}}
	Let X be a RDP Del Pezzo surface of degree $2$ over a perfect field $k$. If the anti-canonical morphism $\kappa : X \rightarrow \bP_k^2$ is purely inseparable, then $X$ is $k$-unirational.
\end{cor}
\begin{proof}
	By Proposition {\ref{Prop:Kawakami}}, we have a dominant morphism $g : \bP_k^2 \rightarrow X$ such that $\kappa \circ g$ is a succession of $n$ Frobenius morphisms of $\bP_k^2$. Since $k$ is perfect, the Frobenius morphism of $\bP_k^2$ is an isomorphism. Therefore, we have the dominant morphism $g \circ ({\rm{Fr}}^n)^{-1} : \bP_k^2 \rightarrow X$ over $k$. This means $X$ is $k$-unirational.
\end{proof}
From the above, Theorem {\ref{MainThm2}} (1) holds.

\subsection{Spine}
	In the following subsection, we suppose that $\kappa$ is separable. 

	Let $R \subset X$ be the ramification divisor of $\kappa$ and let $B \subset \bP_k^2$ be the branch divisor of $\kappa$. Let $F(x,y,z,w) = w^2 + w \cdot q_2(x,y,z) + q_4(x,y,z)$ be the homogeneous quartic equation which defines $X \subset \Proj\ k[x,y,z,w] = \bP_k(1,1,1,2)$. 

	If the characteristic of $k$ is not $2$, then $B$ is the quartic curve defined by $q_2(x,y,z)^2 - 4q_4(x,y,z) = 0$. Since $X$ is reduced, $B$ has no multiple component and since $B$ has only rational singularities, $B$ is not four lines meeting in a point ({\cite[Proposition 4.6]{HW}}). This quartic has singularities corresponding with the singularities of $X$. For example, in the case of $\ctext{1}, \ctext{2}$ and $\ctext{3}$ of Theorem {\ref{MainThm1}}, $B$ is:
\begin{itemize}
	\item[] a singular quartic with one node if $\ctext{1}$,
	\item[] a singular quartic with one cusp if $\ctext{2}$,
	\item[] two conics intersecting at four points if $\ctext{3}$,
\end{itemize}
according to {\cite{DuVal}}.

	If characteristic of $k$ is $2$, then $B$ is a double quadric curve defined by $q_2(x,y,z)^2 = 0$. Note that $q_2(x,y,z) \not \equiv 0$ since $\kappa$ is separable. The quadric $q_2(x,y,z) = 0$ may be a smooth conic, two lines or a double line. 

	$R$ is important because of the following Proposition:
\begin{prop}[See also {\cite[Lemma 2.5]{STVA}}]\label{Prop4.3}
	Let $f : \Xt \rightarrow X$ be a minimal resolution of $X$ and let $p$ be a $k$-point on $\Xt$. 
$$
\Xt \mapright{f} X \mapright{\kappa} \bP_k^2
$$
Suppose that $p$ does not lie on any $(-2)$curves (Recall that a $(-2)$curve means a $(-2)$curve over $\kb$). Then, there exists at most one element in $|-\Kxt|$ which is singular at $p$. There is such an element if and only if $f(p)$ lies on $R$.
\end{prop}
\begin{proof}
	If $f(p) \notin R$, then $\kappa$ is \'etale at $f(p)$ and the image of elements in $|-\Kx|$ are lines $\subset \bP_k ^2$. Therefore, anti-canonical curves on $X$ which pass through $f(p)$ are smooth at $f(p)$. Since $f^* \Kx = \Kxt$, there is no anti-canonical curve on $\Xt$ which is singular at $p$. 

	Assume that $f(p) \in R$. In neighborhood of $f(p)$, $X$ is isomorphic to $\Spec \ k[x,y,w]/F$, where $F$ is a polynomial of degree at most $4$ in $k[x,y,w]$.

	If $\ch (k) \neq 2$, we can assume that $f(p) = (0,0,0)$ and $F(x,y,w) = w^2 + q(x,y)$. Let $C$ be an anti-canonical curve on $X$ which passes through $f(p)$. Then $\kappa_{*} C$ is a line, denoted by $ax + by = 0$. $C$ is singular at $f(p)$ if and only if $w^2 + q(-bt, at) = 0 \subset \Spec \ k[t,w]$ is singular at $(0,0)$. This means that the line $ax + by = 0$ is $q_x(0,0)x + q_y(0,0)y = 0$, where $q_x$ and $q_y$ are partial derivatives of $q$ with respect to $x, y$. Thus there exists a unique anti-canonical curve with a singular point $p$. Note that this is a total transform of the tangent line of $B$ under $\kappa \circ f(p)$.

	If $\ch (k) = 2$, we can assume that $f(p)=(0,0,0)$ and $F(x,y,w)=w^2 + q_2(x,y)w + q_4(x,y)$. Similarly as the case of $\ch(k) \neq 2$, let $C$ be an anti-canonical curve on $X$ which passes through $f(p)$ and $ax + by = 0$ be the line $\kappa_* C$. $C$ is singular at $f(p)$ if and only if $w^2 + q_2(bt,at)w + q_4(bt,at)=0 \subset \Spec\ k[t,w]$ is singular at $(0,0)$. This means that the line $ax + by = 0$ is ${q_4}_x(0,0)x + {q_4}_y(0,0)y = 0$. Similarly as the case $\ch(k) \neq 2$, there exists a unique anti-canonical curve with a singular point $p$.
\end{proof}

\begin{definition}
	Let $p$ be a $k$-point on $\Xt$ which does not lie on any $(-2)$curves. Suppose that $f(p) \in R$. We call the element in $|{-\Kxt}|$ passing through $p$ as a singular point the {\em{spine}} of $\Xt$ at $p$.
\end{definition}
\subsection{Blowing up}
	For the proof of Theorem {\ref{MainThm2}}, we consider blowing ups of $\Xt$ at a $k$-point. 
\begin{lem}{\label{Lem4.6}}
	Let $p$ be a $k$-point on $\Xt$ and $\pi : \Xtp \rightarrow \Xt$ be the blowing up of $\Xt$ with a center $p$. Then $\Xtp$ is also a weak Del Pezzo surface if and only if $p$ does not lie on any $(-2)$curves.
\end{lem}
\begin{proof}
This follows from Proposition {\ref{Prop:Demazure}} and {\ref{Prop:HW}}.
\end{proof}
	For Proposition {\ref{Prop4.7}}, we prepare following lemma.
\begin{lem}[See also {\cite[Lemma 2.1]{STVA}}]{\label{Lem4.5}}
	Let $C$ be an element in $|-\Kxt|$. If $f_*C$ is reducible over $\kb$, then $C$ decomposes over $\kb$ as following:
	$$C = C_1 + C_2 + \sum_{i} F_i,$$
where $C_1, C_2$ are $(-1)$curves on $\Xt$ and $F_i$ are $(-2)$curves on $\Xt$ which may not be distinct. 
\end{lem}
\begin{proof}
	Since $C \cdot (-\Kxt) = 2$ and $-\Kxt$ is nef, we have the following two cases.
\begin{itemize}
	\item[] $C = \Cp + \sum_{i} F_i$ or
	\item[] $C = C_1 + C_2 + \sum_{i} F_i$,
\end{itemize}
where $\Cp, C_1, C_2$ and $F_i$ are irreducible curves and satisfy $\Cp \cdot (-\Kxt) = 2$, $C_1 \cdot (-\Kxt) = C_2 \cdot (-\Kxt) = 1$ and $F_i \cdot (-\Kxt) = 0$. By Lemma {\ref{Lem:exccurves}}, $C_1,C_2$ are $(-1)$curves and $F_i$ are $(-2)$curves. By Corollary {\ref{Cor:RDPDP<->weakDP}}, $f_*F_i$ is points. Therefore, if $C = \Cp + \sum_{i} F_i$, then $f_*C = f_*\Cp$. This is not reducible. Thus $C = C_1 + C_2 + \sum_{i} F_i$.
\end{proof}

\begin{prop}[See also{\cite[Theorem 2.9]{STVA}}]{\label{Prop4.7}}
	Let $p$ be a $k$-point which does not lie on any $(-2)$curves and $\pi : \Xtp \rightarrow \Xt$ be the blowing up at $p$. (By Lemma {\ref{Lem4.6}}, $\Xtp$ is a weak Del Pezzo surface of degree $1$.)
	$$\Xtp \mapright{\pi} \Xt \mapright{f} X$$
For an irreducible divisor $C$ on $\Xtp$, the following two conditions are equivalent:
\begin{itemize}
	\item[(i)] $C$ is a $(-2)$curve on $\Xtp$.
	\item[(ii)] $C$ is one of the following:
		\begin{itemize}
			\item[(1)] $C$ is a total transform of a $(-2)$curve on $\Xt$ under $\pi$.
			\item[(2)] $C$ is a strict transform of a $(-1)$curve passing through $p$ under $\pi$.
			\item[(3)] $f(p) \in X$ lies on the ramification divisor $R$ of $\kappa$ and $C$ is the strict transform of a component of the spine at $p$ under $\pi$.
		\end{itemize}
\end{itemize}
\end{prop}
\begin{proof}
	Assume that (ii) holds. If $C$ is (1) or (2) of (ii), then $C$ is clearly a $(-2)$curve on $\Xtp$. Suppose that $f(p) \in R$ and let $S$ be the spine at $p$. If $f_*(S)$ is reducible over $\kb$, then components of $S$ are $(-1)$curves and $(-2)$curves by Lemma {\ref{Lem4.5}}. Then both of the $(-1)$curves pass through $p$ since $S$ is singular at $p$. Thus the strict transforms of the components are all $(-2)$curves on $\Xtp$. If $f_*(S)$ is irreducible over $\kb$, there exists an irreducible curve $\Sp$ and $(-2)$curves $F_i$ such that:
	$$S = \Sp + \sum_{i} F_i.$$
Then $\pi^* \Sp = \pi_*^{-1} \Sp + n E$, where $\pi_*^{-1} \Sp$ is a strict transform of $\Sp$. Since $S$ is singular at $p$ and $p$ lies on no $(-2)$curve, we have $n \geq 2$. Thus $\pi_*^{-1} \Sp \cdot (-\Kxtp) = 2 - n \leq 0$. This means $n = 2$ and $\pi_*^{-1} \Sp$ is a $(-2)$curve.

	Conversely, let $C$ be a $(-2)$curve on $\Xtp$. By $C \cdot (-\Kxtp) = 0$, we have $\dim H^0 (C, \sh{C}(-\Kxtp)) = 1$. Thus, the long exact sequence associated to the sequence
	$$0 \rightarrow \sh{\Xtp}(-\Kxtp -C) \longrightarrow \sh{\Xtp}(-\Kxtp) \longrightarrow \sh{C}(-\Kxtp) \rightarrow 0$$
induces $\dim H^0(\Xtp, \sh{\Xtp}(-\Kxtp - C)) = 1$ since $\dim H^0(\Xtp,\sh{\Xtp}(-\Kxtp)) = 2$ by Proposition {\ref{dim of anticano}} and $\sh{\Xtp}(-\Kxtp)$ has no fixed component. This means there exists a unique element in $|-\Kxtp|$ which contains $C$ as a component, denoted by $D$. 

	If $D$ contains $E := \pi^{-1}(p)$, we have $f(p) \in R$ and $\pi_* D$ is the spine at $p$ since $\pi_* D \in |-\Kxt|$ and $(D - E) \cdot E = 2$. Thus $C$ is in Case (3) of (ii). If $D$ does not contain $E$, then $C \cdot E = 0$ or $1$. Since $C \cdot (-\Kxtp) = 0$, we have $\pi_*C \cdot (-\Kxt) = 0$ or $1$. Thus (1) or (2) holds.  
\end{proof}
\subsection{Main Theorem}
We show Theorem {\ref{MainThm2}}. At first, we define a generalized Eckardt point.

\begin{lem}[See also{\cite[Lemma 2.2]{STVA}}]{\label{Lem4.8}}
	Let $p$ be a $k$-point on $\Xt$. If there exist four $(-1)$curves passing through $p$, then the sum of the four $(-1)$curves is linearly equivalent to $-2\Kxt$. In particular, there are at most four $(-1)$curves which pass through $p$.
\end{lem}
\begin{proof}
	Suppose that there exist four $(-1)$curves $E_1, E_2, E_3$ and $E_4$ which pass through $p$. $(E_i + E_j) \notin |-\Kxt|$ $(i \neq j)$ since $(E_i + E_j) \cdot E_k = 2$ $(k \neq i,j)$. Thus $E_i \cdot E_j = 1$ $( i \neq j)$ hold by Lemma {\ref{Lem3.6}}. Therefore $(-2\Kxt - E_1 - E_2 - E_3)^2 = (-2\Kxt - E_1 - E_2 - E_3) \cdot \Kxt = -1$. This means $(-2\Kxt - E_1 - E_2 - E_3)$ is linearly equivalent to a pre$(-1)$curve. Since $(-2\Kxt - E_1 - E_2 - E_3) \cdot E_4 = -1$, we have $(E_1 + E_2 + E_3 + E_4) \in |-2\Kxt|$. Since $E \cdot (-2\Kxt) = 2$ for any $(-1)$curves $E$, there is no other $(-1)$curve passing through $p$.
\end{proof}
\begin{definition}{\label{Def:geneEckpt}}
	A {\em{generalized Eckardt point}} is a point on $\Xt$ contained in four $(-1)$curves.
\end{definition}
	
\begin{thm}[See also{\cite[Theorem 3.1]{STVA}}]{\label{Thm4.10}}
	Let $\Xt$ be a weak Del Pezzo surface of degree $2$ over a field $k$. Let $p$ be a $k$-point on $\Xt$ which is not a generalized Eckardt point and let $n$ be the number of $(-1)$curves which pass through $p$ $(0 \leq n \leq 3)$. If $f(p)$ does not lie on the ramification divisor $R$ of $\kappa$, there exists a non-constant morphism $\bP_k^1 \rightarrow \Xt$ such that:
\begin{itemize}
	\item[] the image is singular at $p$ if $n = 0$ or $1$;
	\item[] the image passes through $p$ if $n = 2$;
	\item[] the image is a $(-1)$curve defined over $k$ if $n = 3$.
\end{itemize}
\end{thm}
\begin{proof}
	Let $\pi : \Xtp \rightarrow \Xt$ be the blowing up at $p$ and $E := \pi^{-1}(p)$. Since $f(p)$ does not lie on $R$, $f(p)$ is not a singular point of $X$. Thus $\Xtp$ is a weak Del Pezzo surface of degree $1$ by Lemma {\ref{Lem4.6}}. Since $(-2\Kxtp - E)^2 = (-2\Kxtp - E) \cdot \Kxtp = -1$, there exists a pre$(-1)$curve $D$ defined over $k$ in $|-2\Kxtp - E|$ by Corollary {\ref{Cor3.5}}. Let $D_1$ be the prime divisor defined by Lemma {\ref{Lem3.2}} and let $D_0 := D - D_1$. We show $D_1 \notin |-\Kxtp|$. If $D_1 \in |-\Kxtp|$, then $D_0 \in |-\Kxtp - E|$ and thus $D_0 \cdot E = 2$. Since $\pi_* D_0 \in |-\Kxt|$, $\pi_* D_0$ is the spine at $p$. This is impossible since $f(p) \notin R$. Therefore $D_1 \notin |-\Kxtp|$. By Lemma {\ref{Lem:exccurves}}, $D_1$ is a $(-1)$curve defined over $k$. In particular, $D_1 \cong \bP_k^1$.

	We show that $D_1 \cdot E = 3 - n$ where $n$ is the number of $(-1)$curves which pass through $p$. Let $E_i$ be the $n$ $(-1)$curves on $\Xt$ passing through $p$ $(1 \leq i \leq n)$. By Proposition {\ref{Prop4.7}}, 
	$$D_0 = \sum_{i = 1}^n r_i \pi_*^{-1}E_i + \sum_j \pi^*F_j,$$
where $r_i$ are non-negative integers, $\pi_*^{-1}E_i$ are the strict transforms of $E_i$ and $F_j$ are $(-2)$curves on $\Xt$. Since $D_0 \cdot E = \sum_i r_i$, it suffices to show that $r_i = 1$ for all $i$. Since $D \cdot \pi_*^{-1}E_i = (-2\Kxtp - E) \cdot (\pi^*E_i - E) = -1$, the $(-2)$curves $\pi_*^{-1}E_i$ are contained in $D$ as components. Thus we have $r_i \geq 1$ for all $i$. Suppose that $r_1 \geq 2$. From $D = -2\pi^*\Kxt - 3E = D_1 + r_1 \pi_*^{-1}E_1 + \sum_{i =2}^n r_i \pi_*^{-1}E_i + \sum_j \pi^*F_j$, we deduce the equality
\begin{equation}{\label{equation4.2}}
	2\pi^*(-\Kxt - E_1) = E + D_1 + (r_1 - 2) \pi_*^{-1}E_1 + \sum_{i =2}^n r_i \pi_*^{-1}E_i + \sum_j \pi^*F_j
\end{equation}
of effective divisors of $\Xtp$. Since $(-\Kxt - E_1)^2 = (-\Kxt - E_1) \cdot \Kxt = -1$, we have that $(-\Kxt - E_1)$ is a pre$(-1)$curve on $\Xt$. Let $C$ be the $(-1)$curve contained in $(-\Kxt - E_1)$. Since $f (C + E_1) \in |-\Kx|$ and $f(p) \notin R$, we have $(C + E_1)$ is not singular at $p$ by Proposition {\ref{Prop4.3}}. This means $p$ does not lie on $C$. Thus $\pi^* (-\Kxt - E_1)$ does not contain $E$. This contradicts the equality ({\ref{equation4.2}}) and thus $r_i = 1$ for all $i$. Since $D_1 \cdot E = 3 - n$, the non-constant morphism $\pi|_{D_1} : D_1 \rightarrow \Xt$ is the desired morphism. 
\end{proof}
	The following theorem gives a necessary and sufficient condition of unirationality.
\begin{thm}[See also{\cite[Theorem 3.2]{STVA}}]{\label{Thm4.11}}
	If there exists a non-constant morphism $\rho : \bP_k^1 \rightarrow \Xt$ which image is not contained in $f^* R$, then $\Xt$ is $k$-unirational.
\end{thm}
\begin{proof}
	Let $\eta$ be the generic point of the image of $\rho$ and $\varphi : \Xt \times_k k(\eta) \rightarrow \Xt$ be the projection. Then there exists a $k(\eta)$-point $\etap$ on $\Xt \times_k k(\eta)$ which satisfies $\varphi(\etap) = \eta$. Since $\eta$ is not contained in two $(-1)$curves nor $f^* R$, so is $\etap$. Thus there exists a non-constant morphism $\tau : \bP_{k(\eta)}^1 \rightarrow \Xt \times_k k(\eta)$ by Theorem {\ref{Thm4.10}}.
$$
\xymatrix{
	 & \bP_{k(\eta)}^1 \ar[r] \ar[d]^{\rho \times k(\eta)} & \bP_k^1 \ar[d]^{\rho} \\
	\bP_{k(\eta)}^1 \ar[r]^-{\tau} & \Xt \times_k k(\eta) \ar[r]^-{\varphi} & \Xt \\
}
$$
Since $\bP_{k(\eta)}^1$ is birationally equivalent to $\bP_k^2$ over $k$, it suffices to show that the composition $\varphi \circ \tau$ is dominant. Since $\eta \in \Imr(\varphi \circ \tau)$, we have $\overline{\Imr(\rho)}^{Zar} \subset \overline{\Imr(\varphi \circ \tau)}^{Zar}$, where $\overline{\ \cdot\ }^{Zar}$ is Zariski closure. Further, $\Imr(\tau) \not\subset \overline{\Imr(\rho \times k(\eta))}^{Zar}$ since $\Imr(\tau)$ is singular at $\etap$. This means $\overline{\Imr(\rho)}^{Zar} \subsetneq \overline{\Imr(\varphi \circ \tau)}^{Zar}$. Since $\overline{\Imr(\varphi \circ \tau)}^{Zar}$ is irreducible, we have $\varphi \circ \tau$ is dominant and thus $\Xt$ is $k$-unirational.
\end{proof}
\begin{proof}[Proof of Theorem {\ref{MainThm2}}]
If (1), Proposition {\ref{Prop:Kawakami}} induces $k$-unirationality. If (2), Theorem {\ref{Thm4.10}} and {\ref{Thm4.11}} induce $k$-unirationality.
\end{proof}
\section{Uniraionality over finite fields}
	In this section, we show Theorem {\ref{MainThm3}}. Since geometrically rational surfaces over a finite field has at least one $k$-point (Proposition {\ref{Prop:WeylConj}}), it suffices to consider $\ctext{1}, \ctext{2}$ and $\ctext{3}$ of Theorem {\ref{MainThm1}}. 

	In Section 5.1, we give a necessary condition for $k$-unirationality: if $X$ has enough $k$-points outside the ramification divisor $R$, then it is $k$-unirational. In this subsection, we do not need to suppose that $k$ is finite. In Section 5.2, we give a lower limit of the number of $k$-points of $X$ over a finite field. For calculating this, we use the computer program SageMath{\cite{Sage}}. The functions are contained in the end of the source file in the arXiv posting. In Section 5.3, we show Theorem {\ref{MainThm3}}.
\subsection{The number of ($-$1)curves}
Let $X$ be a RDP Del Pezzo surface of degree $2$ over a field $k$. Suppose that $X$ is $\ctext{1}, \ctext{2}$ or $\ctext{3}$ of Theorem {\ref{MainThm1}}. The following two lemmas hold over any fields.
\begin{lem}{\label{Lem5.2}}
	$\Xt$ has just $n$ $(-1)$curves which do not intersect any $(-2)$curves, where
$$
n = 
\begin{cases}
	32 & $if $\ctext{1}, \\
	20 & $if $\ctext{2}, \\
	8 & $if $\ctext{3}.
\end{cases}
$$
\end{lem}
\begin{proof}
	In the case of $\ctext{1}$, let $F$ be the unique $(-2)$curve. By Proposition {\ref{Prop3.7}}, we have
$$
	\# \{ D : {\rm{pre}}(-1){\rm{curve}}\ |\ D\cdot F = 1\} = 12.
$$
Further, for a pre$(-1)$curve $D$, 
$$
	D \cdot F = 1 \Leftrightarrow (D + F) \cdot F = -1
$$
holds. Therefore
$$
	\# \{ D : {\rm{pre}}(-1){\rm{curve}}\ |\ D\cdot F = -1\} = 12.
$$
Since there are $56$ pre$(-1)$curves by Corollary {\ref{Cor:56pre-1curves}}, the number of $(-1)$curves which do not intersect $F$ is $56 - 12 - 12 = 32$.

	In the case of $\ctext{2}$, let $F_1$ and $F_2$ be the two $(-2)$curves. Similarly as $\ctext{1}$, 
\begin{itemize}
	\item[] $\# \{ D : {\rm{pre}}(-1){\rm{curve}}\ |\ D\cdot F_1 = 1\} = 12$,
	\item[] $\# \{ D : {\rm{pre}}(-1){\rm{curve}}\ |\ D\cdot F_1 = -1\} = 12$,
	\item[] $\# \{ D : {\rm{pre}}(-1){\rm{curve}}\ |\ D\cdot F_1 = 0\} = 32$.
\end{itemize}
By Proposition {\ref{Prop:A2case}}, we have
\begin{itemize}
	\item[] $\# \{ D : {\rm{pre}}(-1){\rm{curve}}\ |\ D\cdot F_1 = 0, D \cdot F_2 = 1\} = 6$,
	\item[] $\# \{ D : {\rm{pre}}(-1){\rm{curve}}\ |\ D\cdot F_1 = -1, D \cdot F_2 = 1\} = 6$.
\end{itemize}
Further, for a pre$(-1)$curve $D$,
$$
	D \cdot F_1 = -1\ {\rm{and}}\ D \cdot F_2 = 1 \Leftrightarrow (D + F_1) \cdot F_1 = 0\ {\rm{and}}\ (D + F_1) \cdot F_2 = -1
$$
holds. This induces
\begin{itemize}
	\item[] $\# \{ D : {\rm{pre}}(-1){\rm{curve}}\ |\ D\cdot F_1 = 0, D \cdot F_2 = -1\} = 6.$
\end{itemize}
Thus the number of $(-1)$curves which do not intersect $F_1, F_2$ is $32 - 6 - 6 = 20$.

	In the case of $\ctext{3}$, let $F_1$, $F_2$, $F_3$ and $F_4$ be the $(-2)$curves. A pre$(-1)$curve $D$ satisfies
$$
		D \cdot F_i = 1 \Leftrightarrow (D + F_i) \cdot F_i = -1
$$
and $D \cdot F_j = (D - F_i) \cdot F_j$ $(i \neq j)$. This induces
\begin{eqnarray}{\label{equation5.1}}
	\# \{ D : {\rm{pre}}(-1){\rm{curve}}\ |\ D\cdot F_{i_1} = r_1, D \cdot F_{i_2} = r_2,  D \cdot F_{i_3} = r_3,  D \cdot F_{i_4} = r_4 \} \nonumber \\
	= \# \{ D : {\rm{pre}}(-1){\rm{curve}}\ |\ D\cdot F_{i_1} = -r_1, D \cdot F_{i_2} = r_2,  D \cdot F_{i_3} = r_3,  D \cdot F_{i_4} = r_4 \}
\end{eqnarray}
where $\{ i_1, i_2, i_3, i_4\} = \{ 1,2,3,4\}$ and $r_1,r_2,r_3,r_4 = 1,0,-1$. On the other hand, by Lemma {\ref{Lem:(-1)curveofLem}}, we have
\begin{equation}{\label{equation5.2}}
	\# \{ D : {\rm{pre}}(-1){\rm{curve}}\ |\ D\cdot F_1 = D \cdot F_2 = D \cdot F_3 = D \cdot F_4 = 1 \} = 0.
\end{equation}
Since $\Xt$ has no $(-1)$curve of Lemma {\ref{Lem3.15}}, we have
\begin{equation}{\label{equation5.3}}
	\# \{ D : {\rm{pre}}(-1){\rm{curve}}\ |\ D\cdot F_1 = D \cdot F_2 = D \cdot F_3 = 1,  D \cdot F_4 = 0 \} = 0.
\end{equation}
By Proposition {\ref{Prop3.8}}, ({\ref{equation5.2}}) and ({\ref{equation5.3}}), we have
\begin{equation}{\label{equation5.4}}
	\# \{ D : {\rm{pre}}(-1){\rm{curve}}\ |\ D\cdot F_1 = D \cdot F_2 = 1, D \cdot F_3 = D \cdot F_4 = 0 \} = 2.
\end{equation}
By ({\ref{equation5.1}}) and ({\ref{equation5.4}}), we have
\begin{equation}
	\# \{ D : {\rm{pre}}(-1){\rm{curve}}\ |\ D\cdot F_1 = 1, D \cdot F_2 = -1, D \cdot F_3 = D \cdot F_4 = 0 \} = 2.
\end{equation}
Therefore,
\begin{equation}{\label{equation5.5}}
	\# \{ D : {\rm{pre}}(-1){\rm{curve}}\ |\ D\cdot F_1 = 1, D \cdot F_2 = D \cdot F_3 = D \cdot F_4 = 0 \} = 12 - 2 \times 6 = 0.
\end{equation}
By ({\ref{equation5.1}}), ({\ref{equation5.2}}), ({\ref{equation5.3}}), ({\ref{equation5.4}}) and ({\ref{equation5.5}}), we have
\begin{itemize}
	\item[] $\# \{ D : {\rm{pre}}(-1){\rm{curve}}\ |\ \exists i,j, D\cdot F_i = D \cdot F_j = 1 \} = 2 \times _4$C$_2 = 12$,
	\item[] $\# \{ D : {\rm{pre}}(-1){\rm{curve}}\ |\ \exists i,j, D\cdot F_i = D \cdot F_j = -1 \} = 2 \times _4$C$_2 = 12$,
	\item[] $\# \{ D : {\rm{pre}}(-1){\rm{curve}}\ |\ \exists i,j, D\cdot F_i = 1, D \cdot F_j = -1 \} = 2 \times _4$P$_2 = 24$,
\end{itemize}
where $_n$C$_m$ and $_n$P$_m$ are $m$-combination of $n$ and $m$-permutation of $n$, respectively. Therefore, the number of $(-1)$curves which do not intersect any $(-2)$curves is $56 - 12 - 12 - 24 = 8$.
\end{proof}
\begin{lem}[See also{\cite[Lemma 3.4]{STVA}}]{\label{Lem5.3}}
	Let $R$ be the ramification divisor of the anti-canonical morphism $\kappa$. Suppose that $X$ has at least $n$ $k$-points which do not lie on $R$ where
$$
n = 
\begin{cases}
	9 & $if $\ctext{1}, \\
	6 & $if $\ctext{2}, \\
	3 & $if $\ctext{3}.
\end{cases}
$$
Then one of the $k$-points is either contained in a $(-1)$curve defined over $k$ or not a generalized Ekcardt point. In particular, $X$ is $k$-unirational by Theorem {\ref{MainThm2}}.
\end{lem}
\begin{proof}
	If there is a $(-1)$curve which contains two of the $n$ $k$-points, then the $(-1)$curve is defined over $k$. Thus we can assume that each $(-1)$curves contains at most one of the $n$ $k$-points. By Lemma {\ref{Lem4.8}}, the sum of four $(-1)$curves which pass through a generalized Eckardt point is in $|-2\Kxt|$. Since such $(-1)$curves do not intersect any $(-2)$curves, the number of generalized Eckardt points is no more than a quarter of the number of $(-1)$curves which do not intersect any $(-2)$curves. Therefore, by Lemma {\ref{Lem5.2}}, one of the $n$ $k$-points is not a generalized Eckardt point.
\end{proof}
\subsection{The number of rational points}
	Suppose that $k$ is a finite field. Over a finite field, the number of $k$-points of $\Xt$ can be calculated by the following theorem:
\begin{prop}[{\cite{Weil}}, {\cite[Theorem 24.1]{Manin}}]{\label{Prop:WeylConj}}
	Let $V$ be a smooth projective surface over a finite field $k$ with $q$ elements. Suppose that $V\times_k \kb$ is $\kb$-rational. Then,
	$$\# V(k) = q^2 + q\ \Tr F^* + 1$$
where $F$ is the Frobenius map in Galois group $\Gal(\kb/k)$ and $F^*$ is the action of $F$ on $\Pic(V \times \kb)$. In particular, $\# V(k) \equiv 1 \mod q$ and thus $\# V(k) \neq 0$
\end{prop}
	We consider candidates for $\Tr F^*$. Let $\omega_{\Xt} \in \Pic(\Xt)$ be the canonical sheaf. The orthogonal component $(\omega_{\Xt})^{\perp}$ in $\Pic(\Xt \times_k \kb)$ is a root lattice of type $E_7$ and thus the Weil group $W(E_7)$ coincides with the group of lattice automorphisms of $\Pic(\Xt \times_k \kb)$ which fix $\omega_{\Xt}$. Since $F^*$ preserves $\omega_{\Xt}$ and intersection numbers, we have $F^* \in W(E_7)$.

	The class of a $(-2)$curve $\in \Pic(\Xt \times_k \kb)$ corresponds to a root in the lattice. Thus, for example, $F^*$ fixes a root if $\ctext{1}$. In this way, we search for candidates for $F^*$ and compute traces of these $F^*$ in Lemma {\ref{Lem5.6}}. As a preparation, we show the following lemma.
\begin{lem}{\label{Lem5.5}}
Let $\Delta$ be the set of roots of $E_7$,
$$\Delta_2 := \{ (r_1, r_2)\ |\ r_1,r_2 \in \Delta,\ r_1 \cdot r_2 = 1\}\ and$$
$$\Delta_3 := \left\{ (r_1,r_2, r_3, r_4) \left|
\begin{array}{l}
	r_1, r_2, r_3, r_4 \in \Delta \\
	r_i \cdot r_j = 0\ (1 \leq i\neq j \leq 4) \\
	\exists g \in W(E_7)\ (g(r_1) = r_2, g(r_2) = r_3, g(r_3) = r_4, g(r_4) = r_1)
\end{array}
\right.\right\}.$$
Then the following statements hold: 
\begin{itemize}
	\item[(1)] The natural action $W(E_7) \curvearrowright \Delta$ is transitive.
	\item[(2)] The natural action $W(E_7) \curvearrowright \Delta_2$ is transitive.
	\item[(3)] The natural action $W(E_7) \curvearrowright \Delta_3$ is transitive.
\end{itemize}
\end{lem}
\begin{proof}
	The transitivities follows from computation by SageMath. The functions are contained in the end of the source file in the arXiv posting.
\end{proof}
\begin{lem}[See also{\cite[Lemma 4.1]{STVA}}]{\label{Lem5.6}}
	Let $\rho : W(E_7) \rightarrow {\rm{Aut}}(\Pic(\Xt \times_k \kb))$ be the natural representation. Let $r$ be a root and let $\delta_2$ and $\delta_3$ be elements in $\Delta_2$ and $\Delta_3$, respectively. Then,
\begin{itemize}
	\item[(1)] $\left\{ \Tr (\rho(g_1)) \left|
\begin{array}{l}
	g_1 \in W(E_7) \\
	\rho(g_1)(r) = r
\end{array}
\right.\right\} = \{ -4, -2, -1, 0, 1, 2, 3, 4, 5, 6, 8\}$
	\item[(2)] $\left\{ \Tr (\rho(g_2)) \left|
\begin{array}{l}
	g_2 \in W(E_7) \\
	g_2$ is a transposition on $\delta_2
\end{array}
\right.\right\} = \{ -4, -2, -1, 0, 1, 2\}.\\$
	\item[(3)] $\left\{ \Tr (\rho(g_3)) \left|
\begin{array}{l}
	g_3 \in W(E_7) \\
	g_3$ is a cyclic permutation on $\delta_3
\end{array}
\right.\right\} = \{ 0, 2\}$
\end{itemize}
\end{lem}
\begin{proof}
	By Lemma {\ref{Lem5.5}}, these sets do not depend on $r, \delta_2$ and $\delta_3$. The rest follows from SageMath computation. The functions are contained in the end of the source file in the arXiv posting.
\end{proof}
These (1), (2) and (3) induce that
\begin{itemize}
	\item[] if $\ctext{1}$ holds, then $\Tr F^* \in \{ -4, -2, -1, 0, 1, 2, 3, 4, 5, 6, 8\}$,
	\item[] if $\ctext{2}$ holds, then $\Tr F^* \in \{ -4, -2, -1, 0, 1, 2\}$ and
	\item[] if $\ctext{3}$ holds, then $\Tr F^* \in \{ 0, 2 \}$,
\end{itemize}
respectively.

\subsection{Main Theorem}
\begin{proof}[Proof of Theorem {\ref{MainThm3}}]
	If (1), it follows from Theorem {\ref{MainThm1}} and Proposition {\ref{Prop:WeylConj}} that $X$ is $k$-unirational.

	Suppose (2). If $\ctext{1}$, the unique $(-2)$curve has no $k$-point or $(q + 1)$ $k$-points and the image of the $(-2)$curve by $f$ is a $k$-point. If $\ctext{2}$ or $\ctext{3}$, the $(-2)$curves have no $k$-point and the images of the $(-2)$curves by $f$ is not $k$-points. Thus, if $\ctext{1}$, we have $\# \Xt(k) = \# X(k) - 1$ or $\# X(k) + q$. If $\ctext{2}$ or $\ctext{3}$, then $\# \Xt(k) = \# X(k)$. By Proposition {\ref{Prop:WeylConj}}, Lemma {\ref{Lem5.6}},
$$\# X(k) \geq
\begin{cases}
	q^2 - 5q + 1 & $if $\ctext{1}, \\
	q^2 - 4q + 1 & $if $\ctext{2}, \\
	q^2 + 1 & $if $\ctext{3}.
\end{cases}
$$

	Next, we consider about $\# R(k)$, where $R$ is the ramification divisor of $\kappa$. In the case of characteristic of $k \neq 2$, $B$ is 
\begin{itemize}
	\item[] a singular quartic with one node if $\ctext{1}$;
	\item[] a singular quartic with one cusp if $\ctext{2}$;
	\item[] a sum of two conics intersecting at four points if $\ctext{3}$,
\end{itemize}
where $B$ is the branch divisor of $\kappa$. If $\ctext{1}$ or $\ctext{2}$, we have $|\# R_{nor}(k) - (q + 1)| \leq 4 \sqrt{q}$ by Hasse-Weil bound ({\cite{Weil48}}), where $R_{nor}$ is a normalization of $R$, since the genus of $R_{nor}$ is $2$. If $\ctext{3}$, we have $\# R_{nor}(k) = 0$ or $2q + 2$. On the other hand,
$$\# R(k) =
\begin{cases}
	\# R_{nor}(k) \pm 1 & $if $\ctext{1}, \\
	\# R_{nor}(k) & $if $\ctext{2}, \\
	\# R_{nor}(k) & $if $\ctext{3}.
\end{cases}
$$
Thus
$$\# R(k) \leq
\begin{cases}
	4\sqrt{q} + q + 2 & $if $\ctext{1}, \\
	4\sqrt{q} + q + 1 & $if $\ctext{2}, \\
	2q + 2 & $if $\ctext{3}.
\end{cases}
$$
If $\ctext{1}$, $q \geq 9 \Rightarrow \# (X \setminus R)(k) \geq 14$. If $\ctext{2}$, $q \geq 9 \Rightarrow \# (X \setminus R)(k) \geq 24$. If $\ctext{3}$, $q \geq 5 \Rightarrow \# (X \setminus R)(k) \geq 14$. By Lemma {\ref{Lem5.3}}, $X$ is $k$-unirational in each case.

	In the case of the characteristic of $k = 2$, $B$ is a quadric which may not be reducible. Thus $\# R(k) = 1, q + 1$ or $2q + 1$. If $\ctext{1}$, $q \geq 16 \Rightarrow \# (X \setminus R)(k) \geq 144$. If $\ctext{2}$, $q \geq 8 \Rightarrow \# (X \setminus R)(k) \geq 16$. If $\ctext{3}$, $q \geq 4 \Rightarrow \# (X \setminus R)(k) \geq 8$. By Lemma {\ref{Lem5.3}}, $X$ is $k$-unirational in each case.
\end{proof}
\appendix
\section{Classification of RDP Del Pezzo surfaces with one, two or three singularities}
	Let $X$ be a RDP Del Pezzo surface of degree $2$ over a perfect field $k$. The minimal resolution of $X$ is a weak Del Pezzo surface by Corollary {\ref{Cor:RDPDP<->weakDP}}, denote this by $\Xt$. If $X$ has one, two or three singularities and not $\ctext{1}, \ctext{2}$ of Theorem {\ref{MainThm1}}, then we obtain a birational morphism $\Xt \rightarrow \Vt$ by Theorem {\ref{Thm3.11}}, {\ref{Thm3.14}} and {\ref{MainThm3.5}}. In this appendix, we identify each $\Vt$ by singularities of $X$.

\begin{thmA1}
	Let $X$ be a RDP Del Pezzo surface of degree $2$ with two singularities and $\Vt$ be the weak Del Pezzo surface of degree $3$ or $4$ defined in Theorem {\ref{Thm3.11}}. By Proposition {\ref{Prop:singularity type}}, the type of singularities of $X$ is either $2A_1, A_1 + A_2, A_1 + A_3, A_1 + A_4, A_1 + A_5, A_1 + D_4, A_1 + D_5, A_1 + D_6,  2A_2, A_2 + A_3, A_2 + A_4, A_2 + A_5$ or $2A_3$. In these 13 cases, the degree and the singularities of $V$ are respectively as follows, where $V$ is the RDP Del Pezzo surface obtained by collapsing the $(-2)$curves of $\Vt$.

	In the following configurations, vertices $\circ$ and $\bullet$ are the $(-2)$curves on $\Xt$ and the $(-1)$curve(s) $E, \Ep$ in proof of Theorem {\ref{Thm3.11}}, respectively. Each edge represents the intersection of two vertices.

\begin{itemize}
	\item[1.] If $X$ is type $2A_1$, $V$ is a smooth Del Pezzo surface of degree $4$.
$$
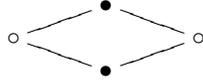
\xymatrix@R = 2pt{
	 & \bullet \ar@{-}[rd] & \\
	\circ \ar@{-}[ru] \ar@{-}[rd] & & \circ \\
	 & \bullet \ar@{-}[ru] &
}
$$
	\item[2.] If $X$ is type $A_1 + A_2$, $V$ is a smooth Del Pezzo surface of degree $4$.
$$
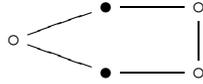
\xymatrix@R = 2pt{
	 & \bullet \ar@{-}[r] & \circ \ar@{-}[dd] \\
	\circ \ar@{-}[ru] \ar@{-}[rd] & & \\
	 & \bullet \ar@{-}[r] & \circ
}
$$
	\item[3.]  If $X$ is type $A_1 + A_3$, $V$ is either (1)a RDP Del Pezzo surface of degree $3$ with $2A_1$ singularities or (2)a RDP Del Pezzo surface of degree $4$ with $A_1$ singularity.
$$(1)
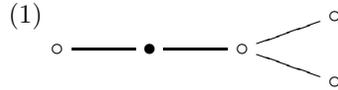
\xymatrix@R = 2pt{
	 & & & \circ \ar@{-}[ld] \\
	\circ \ar@{-}[r] & \bullet \ar@{-}[r] & \circ \ar@{-}[rd]\\
	 & & & \circ
}
$$
or
$$(2)
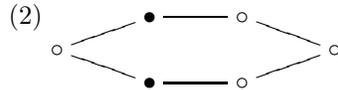
\xymatrix@R = 2pt{
	 & \bullet \ar@{-}[r] & \circ \ar@{-}[rd] \\
	\circ \ar@{-}[ru] \ar@{-}[rd] & & & \circ \ar@{-}[ld]\\
	 & \bullet \ar@{-}[r] & \circ
}
$$
	\item[4.] If $X$ is type $A_1 + A_4$, $V$ is a RDP Del Pezzo surface of degree $4$ with $A_2$ singularity.
$$
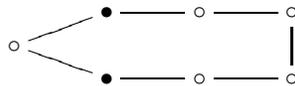
\xymatrix@R = 2pt{
	 & \bullet \ar@{-}[r] & \circ \ar@{-}[r] & \circ \ar@{-}[dd] \\
	\circ \ar@{-}[ru] \ar@{-}[rd] & & &\\
	 & \bullet \ar@{-}[r] & \circ \ar@{-}[r] & \circ
}
$$
	\item[5.] If $X$ is type $A_1 + A_5$, $V$ is a RDP Del Pezzo surface of degree $4$ with $A_3$ singularity.
$$
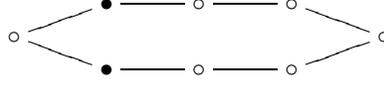
\xymatrix@R = 2pt{
	 & \bullet \ar@{-}[r] & \circ \ar@{-}[r] & \circ \ar@{-}[rd] \\
	\circ \ar@{-}[ru] \ar@{-}[rd] & & & & \circ \ar@{-}[ld]\\
	 & \bullet \ar@{-}[r] & \circ \ar@{-}[r] & \circ
}
$$
	\item[6.] If $X$ is type $A_1 + D_4$, $V$ is a RDP Del Pezzo surface of degree $3$ with $A_3$ singularity.
$$
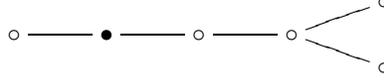
\xymatrix@R = 2pt{
	 & & & & \circ \ar@{-}[ld] \\
	\circ \ar@{-}[r] & \bullet \ar@{-}[r] & \circ \ar@{-}[r] & \circ \ar@{-}[rd]\\
	 & & & & \circ
}
$$
	\item[7.] If $X$ is type $A_1 + D_5$, $V$ is a RDP Del Pezzo surface of degree $3$ with $D_4$ singularity.
$$
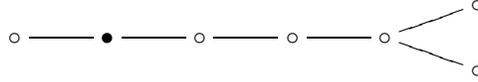
\xymatrix@R = 2pt{
	 & & & & & \circ \ar@{-}[ld] \\
	\circ \ar@{-}[r] & \bullet \ar@{-}[r] & \circ \ar@{-}[r] & \circ \ar@{-}[r] & \circ \ar@{-}[rd]\\
	 & & & & & \circ
}
$$
	\item[8.] If $X$ is type $A_1 + D_6$, $V$ is a RDP Del Pezzo surface of degree $3$ with $D_5$ singularity.
$$
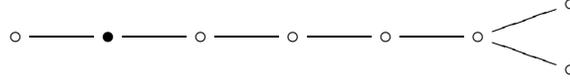
\xymatrix@R = 2pt{
	 & & & & & & \circ \ar@{-}[ld] \\
	\circ \ar@{-}[r] & \bullet \ar@{-}[r] & \circ \ar@{-}[r] & \circ \ar@{-}[r] & \circ \ar@{-}[r] & \circ \ar@{-}[rd]\\
	 & & & & & & \circ
}
$$
	\item[9.] If $X$ is type $2A_2$, $V$ is a smooth Del Pezzo surface of degree $4$.
$$
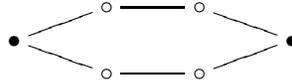
\xymatrix@R = 2pt{
	 & \circ \ar@{-}[r] & \circ \ar@{-}[rd] & \\
	\bullet \ar@{-}[ru] \ar@{-}[rd] & & & \bullet \\
	 & \circ \ar@{-}[r] & \circ \ar@{-}[ru] &
}
$$
	\item[10.] If $X$ is type $A_2 + A_3$, $V$ is a RDP Del Pezzo surface of degree $4$ with $A_1$ singularity.
$$
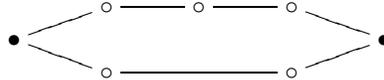
\xymatrix@R = 2pt{
	 & \circ \ar@{-}[r] & \circ \ar@{-}[r] & \circ \ar@{-}[rd] & \\
	\bullet \ar@{-}[ru] \ar@{-}[rd] & & & & \bullet \\
	 & \circ \ar@{-}[rr] & & \circ \ar@{-}[ru] &
}
$$
	\item[11.] If $X$ is type $A_2 + A_4$, $V$ is a RDP Del Pezzo surface of degree $4$ with $A_2$ singularity.
$$
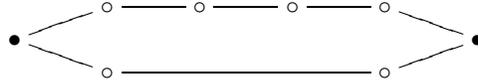
\xymatrix@R = 2pt{
	 & \circ \ar@{-}[r] & \circ \ar@{-}[r] & \circ \ar@{-}[r] & \circ \ar@{-}[rd] & \\
	\bullet \ar@{-}[ru] \ar@{-}[rd] & & & & & \bullet \\
	 & \circ \ar@{-}[rrr] & & & \circ \ar@{-}[ru] &
}
$$
	\item[12.] If $X$ is type $A_2 + A_5$, $V$ is a RDP Del Pezzo surface of degree $4$ with $A_3$ singularity in Case $7$ of Proposition {\ref{Prop:classify degree4}}.
$$
\xymatrix@R = 2pt{
	 & \circ \ar@{-}[r] & \circ \ar@{-}[r] &  \circ \ar@{-}[r] &\circ \ar@{-}[r] & \circ \ar@{-}[rd] & \\
	\bullet \ar@{-}[ru] \ar@{-}[rd] & & & & & & \bullet \\
	 & \circ \ar@{-}[rrrr] & & & & \circ \ar@{-}[ru] &
}
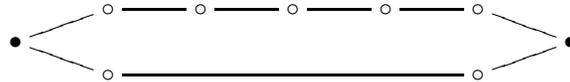
$$
	\item[13.] If $X$ is type $2A_3$, $V$ is a RDP Del Pezzo surface of degree $4$ with $2A_1$ singularity in Case $3$ of Proposition {\ref{Prop:classify degree4}}.
$$
\xymatrix@R = 2pt{
	 & \circ \ar@{-}[r] & \circ \ar@{-}[r] & \circ \ar@{-}[rd] & \\
	\bullet \ar@{-}[ru] \ar@{-}[rd] & & & & \bullet \\
	 & \circ \ar@{-}[r] & \circ \ar@{-}[r] & \circ \ar@{-}[ru] &
}
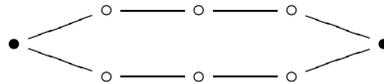
$$
\end{itemize}
\end{thmA1}

\begin{remA2}
	In particular, the unirationality or rationality of $X$ is the following.
\begin{itemize}
	\item[] in cases 1, 9 and 13, $X$ is $k$-unirational if $\Xt$ has a $k$-point;
	\item[] in cases 2 and 3(1), $X$ is $k$-unirational;
	\item[] in cases 3(2) and 5, $X$ is $k$-rational if $\Xt$ has a $k$-point;
	\item[] in cases 4, 6, 7, 8, 10, 11 and 12, $X$ is $k$-rational.
\end{itemize}
\end{remA2}

\begin{proof}[Proof of the case 2, 3 and 6]
we prove two cases.
\begin{itemize}
	\item[2.] Let $F_1, F_2$ and $F_3$ be the three $(-2)$curves such that
$$
\xygraph{
	\circ ([]!{+(0,+.2)} {F_1}) [r]
	\circ ([]!{+(0,+.2)} {F_2}) - [r]
	\circ ([]!{+(0,+.2)} {F_3})
}
$$
Let $D_1, D_2$ be the two pre$(-1)$curves such that $D_i \cdot F_1 = D_i \cdot F_2 = 1$. Since $(D_1 + D_2) \in |-\Kxt - F_1 - F_2|$ from Proposition {\ref{Prop3.8}}, we have $(D_1 + D_2) \cdot F_3 = -1$. Thus we can assume that $D_1 \cdot F_3 = 0$ and $D_2 \cdot F_3 = -1$. Then $D_1$ is a $(-1)$curve by Proposition {\ref{Prop3.4}}. On the other hand, $(D_2 - F_3)$ is a pre$(-1)$curve by Lemma {\ref{Lem:minusF}}. Further
\begin{itemize}
	\item[] $(D_2 - F_3) \cdot F_1 = 1$
	\item[] $(D_2 - F_3) \cdot F_2 = 0$
	\item[] $(D_2 - F_3) \cdot F_3 = 1$
\end{itemize}
Thus $(D_2 - F_3)$ is a $(-1)$curve. Since $D_1 \cdot (D_2 - F_3) = 0$ from Proposition {\ref{Prop3.9}}, the configuration is
$$
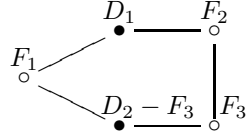
\xygraph{
	\circ ([]!{+(0,+.2)} {F_1})(
	- []!{+(1,.5)} \bullet ([]!{+(0,+.2)} {D_1})
	- [r] \circ ([]!{+(0,+.2)} {F_2})
	- [d] \ ,
	- []!{+(1,-.5)} \bullet ([]!{+(.3,.2)} {D_2 - F_3})
	- [r] \circ ([]!{+(.2,+.2)} {F_3})
}
$$

	\item[3.] Let $F_1, F_2, F_3$ and $F_4$ be the four $(-2)$curves such that
$$
\xygraph{
	\circ ([]!{+(0,+.2)} {F_1}) [r]
	\circ ([]!{+(0,+.2)} {F_2}) - [r]
	\circ ([]!{+(0,+.2)} {F_3}) -[r]
	\circ ([]!{+(0,+.2)} {F_4})
)
}
$$
Let $D_1, D_2$ be the two pre$(-1)$curves such that $D_i \cdot F_1 = D_i \cdot F_3 = 1$. Since $(D_1 + D_2) \in |-\Kxt - F_1 - F_3|$, we have
\begin{itemize}
	\item[] $(D_1 + D_2) \cdot F_2 = -1$
	\item[] $(D_1 + D_2) \cdot F_4 = -1$
\end{itemize}
We can assume that $D_1, D_2$ satisfy one of the following two cases.
\begin{itemize}
	\item[(1)] $D_1 \cdot F_2 = 0$, $D_2 \cdot F_2 = -1$,
	\item[] $D_1 \cdot F_4 = 0$, $D_2 \cdot F_4 = -1$;
	\item[(2)] $D_1 \cdot F_2 = 0$, $D_2 \cdot F_2 = -1$,
	\item[] $D_1 \cdot F_4 = -1$, $D_2 \cdot F_4 = 0$.
\end{itemize}
If (1) holds, then $D_1$ is a $(-1)$curve by Proposition {\ref{Prop3.4}} and one can show that $(D_2 - F_2 - F_3 - F_4)$ is a pre$(-1)$curve by using Lemma {\ref{Lem:minusF}} repeatedly. Since $D_1 \cdot (D_2 - F_2 - F_3 - F_4) = -1$, we have $D_1 = D_2 - F_2 - F_3 - F_4$ and the configuration is 
$$
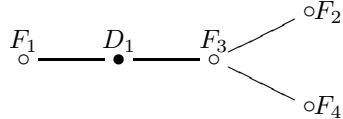
\xygraph{
	\circ ([]!{+(0,+.2)} {F_1}) - [r]
	\bullet ([]!{+(0,+.2)} {D_1}) - [r]
	\circ ([]!{+(0,+.2)} {F_3})(
		- []!{+(1,.5)} \circ ([]!{+(.2,0)} {F_2}),
		- []!{+(1,-.5)} \circ ([]!{+(.2,0)} {F_4})
}
$$

	If (2) holds, then $(D_1 - F_4)$ and $(D_2 - F_2)$ are pre$(-1)$curves by Lemma {\ref{Lem:minusF}}. Further,
\begin{itemize}
	\item[] $(D_1 - F_4) \cdot F_1 = 1$, $(D_2 - F_2) \cdot F_1 = 1$,
	\item[] $(D_1 - F_4) \cdot F_2 = 0$, $(D_2 - F_2) \cdot F_2 = 1$,
	\item[] $(D_1 - F_4) \cdot F_3 = 0$, $(D_2 - F_2) \cdot F_3 = 0$,
	\item[] $(D_1 - F_4) \cdot F_4 = 1$, $(D_2 - F_2) \cdot F_4 = 0$.
\end{itemize}
Thus, $(D_1 - F_4)$ and $(D_2 - F_2)$ are $(-1)$curves by Proposition {\ref{Prop3.4}} and the configuration is
$$
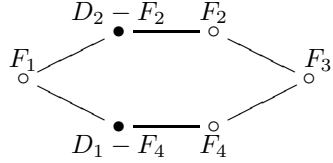
\xygraph{
	\circ ([]!{+(0,+.2)} {F_1})(
	- []!{+(1,.5)} \bullet ([]!{+(0,+.2)} {D_2 - F_2})
		- [r] \circ ([]!{+(0,+.2)} {F_2})
		- []!{+(1,-.5)} \circ ([]!{+(.1,.2)} {F_3}),
	- []!{+(1,-.5)} \bullet ([]!{+(0,-.2)} {D_1 - F_4})
		- [r] \circ ([]!{+(0,-.2)} {F_4})
		- []!{+(1,.5)} \ 
}
$$

	\item[6.] Let $F_1, F_2, F_3, F_4$ and $F_5$ be the five $(-2)$curves such that
$$
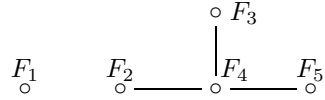
\xygraph{
	\circ ([]!{+(0,+.2)} {F_1}) [r]
	\circ ([]!{+(0,+.2)} {F_2}) - [r]
	\circ ([]!{+(.2,+.2)} {F_4})(
		-[]!{+(0,.8)} \circ ([]!{+(.3,0)} {F_3}),
		-[r] \circ ([]!{+(0,+.2)} {F_5})
)
}
$$
Let $D_1, D_2$ be the two pre$(-1)$curves such that $D_i \cdot F_1 = D_i \cdot F_4 = 1$. Since $(D_1 + D_2) \in |-\Kxt - F_1 - F_4|$, we have
\begin{itemize}
	\item[] $(D_1 + D_2) \cdot F_2 = -1$
	\item[] $(D_1 + D_2) \cdot F_3 = -1$
	\item[] $(D_1 + D_2) \cdot F_5 = -1$
\end{itemize}
We show that it is impossible that $D_1 \cdot F_2 = D_1 \cdot F_3 = D_1 \cdot F_5 = -1$. If $D_1 \cdot F_2 = D_1 \cdot F_3 = D_1 \cdot F_5 = -1$, then $(D_1 - F_2 - F_3 - F_5)$ is a pre$(-1)$curve since $(D_1 - F_2 - F_3 - F_5)^2 = (D_1 - F_2 - F_3 - F_5) \cdot \Kxt = -1$. However, $(D_1 - F_2 - F_3 - F_5) \cdot F_4 = D_1 \cdot F_4 - 3 \leq -2$. This is impossible. Similarly $D_2 \cdot F_2 = D_2 \cdot F_3 = D_2 \cdot F_5 = -1$ is also impossible. Thus we can assume that 
\begin{itemize}
	\item[] $D_1 \cdot F_2 = -1, D_2 \cdot F_2 = 0$
	\item[] $D_1 \cdot F_3 = -1, D_2 \cdot F_3 = 0$
	\item[] $D_1 \cdot F_5 = 0, D_2 \cdot F_5 = -1$
\end{itemize}
Then $(D_1 - F_2 - F_3 - F_4 - F_5)$ and $(D_2 - F_5)$ are pre$(-1)$curves and $(D_1 - F_2 - F_3 - F_4 - F_5) \cdot (D_2 - F_5) = -1$. Therefore $(D_1 - F_2 - F_3 - F_4 - F_5) = (D_2 - F_5)$. Further
\begin{itemize}
	\item[] $(D_2 - F_5) \cdot F_1 = 1$
	\item[] $(D_2 - F_5) \cdot F_2 = 0$
	\item[] $(D_2 - F_5) \cdot F_3 = 0$
	\item[] $(D_2 - F_5) \cdot F_4 = 0$
	\item[] $(D_2 - F_5) \cdot F_5 = 1$
\end{itemize}
Thus the configuration is
$$
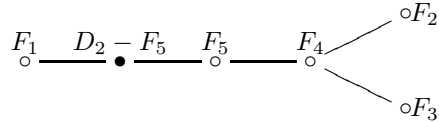
\xygraph{
	\circ ([]!{+(0,+.2)} {F_1}) - [r]
	\bullet ([]!{+(0,+.2)} {D_2 - F_5}) - [r]
	\circ ([]!{+(0,+.2)} {F_5}) - [r]
	\circ ([]!{+(0,+.2)} {F_4})(
		- []!{+(1,.5)} \circ ([]!{+(.2,0)} {F_2}),
		- []!{+(1,-.5)} \circ ([]!{+(.2,0)} {F_3})
}
$$
\end{itemize}
\end{proof}
\begin{thmA3}
	Let $X$ be a RDP Del Pezzo surface of degree $2$ with three singularities and $\Vt$ be the weak Del Pezzo surface of degree at least $3$ defined in Theorem {\ref{Thm3.14}}. By Proposition {\ref{Prop:singularity type}}, the type of singularities of $X$ is either $3A_1, 2A_1 + A_2, 2A_1 + A_3, 2A_1 + D_4, A_1 + 2A_2, A_1 + A_2 + A_3, A_1 + 2A_3$ or $3A_2$. In these 8 cases, the degree and the singularities of $V$ are respectively as follows, where $V$ is the RDP Del Pezzo surface obtained by collapsing the $(-2)$curves of $\Vt$.

	In the following configurations, vertices $\circ$ and $\bullet$ are the $(-2)$curves on $\Xt$ and at most six $(-1)$curves $E_{12}, \Ep_{12}, E_{13}, \Ep_{13}, E_{23}, \Ep_{23}$ in proof of Theorem {\ref{Thm3.14}}, respectively. Each edge represents the intersection of two vertices.
\begin{itemize}
	\item[1.] If $X$ is type $3A_1$, $V$ is either (1)a smooth Del Pezzo surface of degree $3$ or (2)a smooth Del Pezzo surface of degree $8$.
$$
(1)
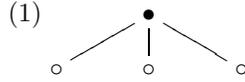
\xymatrix@R = 10pt{
	 & \bullet \ar@{-}[d] \ar@{-}[rd] & \\
	\circ \ar@{-}[ru] & \circ & \circ
}
$$
or
$$
(2)
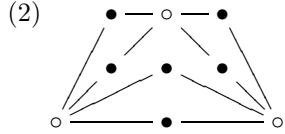
\xymatrix@ = 10pt{
	 & \bullet \ar@{-}[r] & \circ \ar@{-}[r] \ar@{-}[rd] & \bullet \ar@{-}[rdd] & \\
	 & \bullet \ar@{-}[ru] \ar@{-}[ld] & \bullet \ar@{-}[rrd] & \bullet \ar@{-}[rd] & \\
	\circ \ar@{-}[ruu] \ar@{-}[rru] \ar@{-}[rr] & & \bullet \ar@{-}[rr] & & \circ
}
$$
	\item[2.] If $X$ is type $2A_1 + A_2$, $V$ is a smooth Del Pezzo surface of degree $8$.
$$
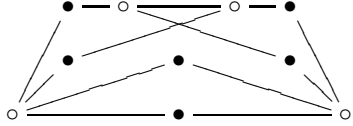
\xymatrix@ = 10pt{
	 & \bullet \ar@{-}[r] & \circ \ar@{-}[rr] \ar@{-}[rrrd] & & \circ \ar@{-}[r] & \bullet \ar@{-}[rdd] & \\
	 & \bullet \ar@{-}[rrru] \ar@{-}[ld] & & \bullet \ar@{-}[rrrd] & & \bullet \ar@{-}[rd] & \\
	\circ \ar@{-}[ruu] \ar@{-}[rrru] \ar@{-}[rrr] & & & \bullet \ar@{-}[rrr] & & & \circ
}
$$
	\item[3.] If $X$ is type $2A_1 + A_3$, $V$ is either (1)a RDP Del Pezzo surface of degree $8$ with $A_1$ singularity or (2)a smooth Del Pezzo surface of degree $7$.
$$(1)
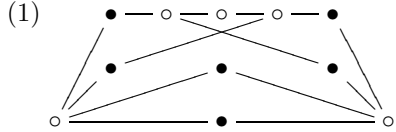
\xymatrix@ = 10pt{
	 & \bullet \ar@{-}[r] & \circ \ar@{-}[r] \ar@{-}[rrrd] & \circ \ar@{-}[r] & \circ \ar@{-}[r] & \bullet \ar@{-}[rdd] & \\
	 & \bullet \ar@{-}[rrru] \ar@{-}[ld] & & \bullet \ar@{-}[rrrd] & & \bullet \ar@{-}[rd] & \\
	\circ \ar@{-}[ruu] \ar@{-}[rrru] \ar@{-}[rrr] & & & \bullet \ar@{-}[rrr] & & & \circ
}
$$
or
$$(2)
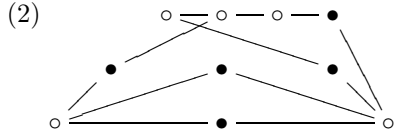
\xymatrix@ = 10pt{
	 &  & \circ \ar@{-}[r] \ar@{-}[rrrd] & \circ \ar@{-}[r] & \circ \ar@{-}[r] & \bullet \ar@{-}[rdd] & \\
	 & \bullet \ar@{-}[rru] \ar@{-}[ld] & & \bullet \ar@{-}[rrrd] & & \bullet \ar@{-}[rd] & \\
	\circ \ar@{-}[rrru] \ar@{-}[rrr] & & & \bullet \ar@{-}[rrr] & & & \circ
}
$$
	\item[4.] If $X$ is type $2A_1 + D_4$, $V$ is a RDP Del Pezzo surface of degree $6$ with $A_2$ singularity.
$$
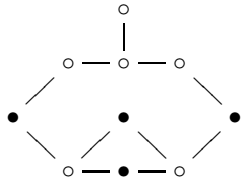
\xymatrix@ = 10pt{
	 & & \circ \ar@{-}[d] & & \\
	 & \circ \ar@{-}[r] & \circ \ar@{-}[r] & \circ \ar@{-}[rd] & \\
	\bullet \ar@{-}[ru] \ar@{-}[rd] & & \bullet \ar@{-}[rd] & & \bullet \\
	 & \circ \ar@{-}[ru] \ar@{-}[r] & \bullet \ar@{-}[r] & \circ \ar@{-}[ru] &
}
$$
	\item[5.] If $X$ is type $A_1 + 2A_2$, $V$ is a smooth Del Pezzo surface of degree $8$.
$$
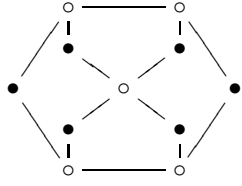
\xymatrix@R = 5pt @C = 10pt{
	 & \circ \ar@{-}[d] \ar@{-}[rr] & & \circ \ar@{-}[rdd] \ar@{-}[d] & \\
	 & \bullet \ar@{-}[rd] & & \bullet & \\
	\bullet \ar@{-}[ruu] \ar@{-}[rdd] & & \circ \ar@{-}[ru] \ar@{-}[rd] & & \bullet \\
	 & \bullet \ar@{-}[d] \ar@{-}[ru] & & \bullet \ar@{-}[d] & \\
	 & \circ \ar@{-}[rr] & & \circ \ar@{-}[ruu]
}
$$
	\item[6.] If $X$ is type $A_1 + A_2 + A_3$, $V$ is a smooth Del Pezzo surface of degree $7$.
$$
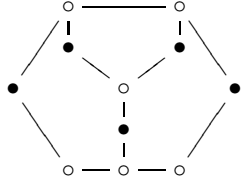
\xymatrix@R = 5pt @C = 10pt{
	 & \circ \ar@{-}[d] \ar@{-}[rr] & & \circ \ar@{-}[rdd] \ar@{-}[d] & \\
	 & \bullet \ar@{-}[rd] & & \bullet & \\
	\bullet \ar@{-}[ruu] \ar@{-}[rdd] & & \circ \ar@{-}[ru] \ar@{-}[d] & & \bullet \\
	 & & \bullet \ar@{-}[d] & & \\
	 & \circ \ar@{-}[r] & \circ \ar@{-}[r] & \circ \ar@{-}[ruu]
}
$$

	\item[7.] If $X$ is type $A_1 + 2A_3$, $V$ is a smooth Del Pezzo surface of degree $6$.
$$
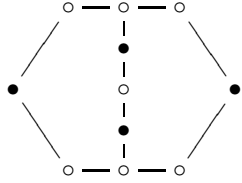
\xymatrix@R = 5pt @C = 10pt{
	 & \circ \ar@{-}[r] & \circ \ar@{-}[d] \ar@{-}[r] & \circ \ar@{-}[rdd] & \\
	 & & \bullet \ar@{-}[d] & & \\
	\bullet \ar@{-}[ruu] \ar@{-}[rdd] & & \circ \ar@{-}[d] & & \bullet \\
	 & & \bullet \ar@{-}[d] & & \\
	 & \circ \ar@{-}[r] & \circ \ar@{-}[r] & \circ \ar@{-}[ruu]
}
$$
	\item[8.] If $X$ is type $3A_2$, $V$ is a smooth Del Pezzo surface of degree $8$.
$$
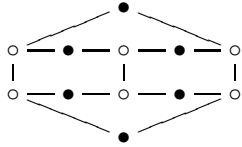
\xymatrix@R = 6pt @C = 10pt{
	 & & \bullet \ar@{-}[rrd] & & \\
	\circ \ar@{-}[rru] \ar@{-}[r] \ar@{-}[d] & \bullet \ar@{-}[r] & \circ \ar@{-}[r] \ar@{-}[d] & \bullet \ar@{-}[r] & \circ \ar@{-}[d] \\
	\circ \ar@{-}[rrd] \ar@{-}[r] & \bullet \ar@{-}[r] & \circ \ar@{-}[r] & \bullet \ar@{-}[r] & \circ \\
	 & & \bullet \ar@{-}[rru] & & 
}
$$
\end{itemize}
\end{thmA3}
\begin{remA4}
	In particular, the unirationalty or rationality of $X$ is the following.
\begin{itemize}
	\item[] in case 1(1), $X$ is $k$-unirational;
	\item[] in cases 1(2), 3(1), 5, 7 and 8, $X$ is $k$-rational if $\Xt$ has a $k$-point;
	\item[] in cases 2, 3(2), 4 and 6, $X$ is $k$-rational.
\end{itemize}
\end{remA4}
\begin{thmA5}
	Let $X$ be a RDP Del Pezzo surface of degree $2$ with one singularity and suppose $X$ is neither $\ctext{1}$ nor $\ctext{2}$.  Let $\Vt$ be the weak Del Pezzo surface defined in the proof of Theorem {\ref{MainThm3.5}}. By Proposition {\ref{Prop:singularity type}}, the type of singularities of $X$ is either $A_2, A_3, A_4, A_5, A_6, A_7, D_4, D_5, D_6, E_6$ or $E_7$. In these 11 cases, the degree and the singularities of $V$ are respectively as follows, where $V$ is the RDP Del Pezzo surface obtained by collapsing the $(-2)$curves of $\Vt$.

	In the following configurations, vertices $\circ$ are the $(-2)$curves on $\Xt$ and vertices $\bullet$ are $\{ D_1, D_2, \cdots, D_6\}$ of Proposition {\ref{Prop:A2case}} if the type is $A_2$, $\{ E_{12}, \Ep_{12}, E_{14}, \Ep_{14}, E_{24}, \Ep_{24} \}$ of Theorem {\ref{MainThm3.5}} if the type is $D_4$ and $\{ E_1, E_2\}$ of Proposition {\ref{Prop:MainProp3.5}} otherwise.
\begin{itemize}
	\item[1.] If $X$ is type $A_2$ and each $(-1)$curves is defined over $k$, then $V$ is a RDP Del Pezzo surface of degree $8$ with $A_1$ singularity.
$$
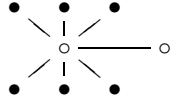
\xymatrix@R = 5pt @C = 8pt{
	\bullet \ar@{-}[rd] & \bullet \ar@{-}[d] & \bullet \ar@{-}[ld] & & \\
	 & \circ \ar@{-}[rr] & & \circ \\
	\bullet \ar@{-}[ru] & \bullet \ar@{-}[u] & \bullet \ar@{-}[lu] & &
}
$$
	\item[2.] If $X$ is type $D_4$, $V$ is a RDP Del Pezzo surface of degree $8$ with $A_1$ singularity.
$$
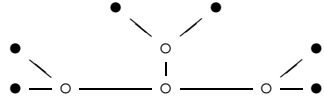
\xymatrix@R = 5pt @C = 8pt{
	 & & \bullet \ar@{-}[rd] & & \bullet \ar@{-}[ld] & & \\
	\bullet \ar@{-}[rd] & & & \circ \ar@{-}[d] & & & \bullet \\
	\bullet \ar@{-}[r] & \circ \ar@{-}[rr] & & \circ \ar@{-}[rr] & & \circ \ar@{-}[r] \ar@{-}[ru] & \bullet
}
$$
	\item[3.] If $X$ is type $A_3$, $V$ is a RDP Del Pezzo surface of degree $4$ with $2A_1$ singularities in Case3 of Proposition {\ref{Prop:classify degree4}}.
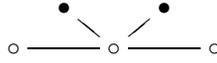
$$
\xymatrix@R = 5pt @C = 8pt{
	 & \bullet \ar@{-}[rd] & & \bullet & \\
	\circ \ar@{-}[rr] & & \circ \ar@{-}[rr] \ar@{-}[ru] & & \circ
}
$$
	\item[4.] If $X$ is type $A_4$, $V$ is a RDP Del Pezzo surface of degree $4$ with $2A_1$ singularities in Case3 of Proposition {\ref{Prop:classify degree4}}.
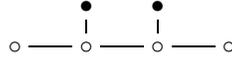
$$
\xymatrix@R = 5pt @C = 16pt{
	 & \bullet \ar@{-}[d] & \bullet \ar@{-}[d]& \\
	\circ \ar@{-}[r] & \circ \ar@{-}[r] & \circ \ar@{-}[r] & \circ
}
$$
	\item[5.] If $X$ is type $A_5$, $V$ is either (1)a RDP Del Pezzo surface of degree $3$ with $2A_2$ singularities or (2)a RDP Del Pezzo surface of degree $4$ with $3A_1$ singularities.
$$
(1)
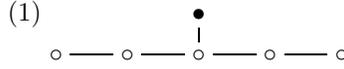
\xymatrix@R = 5pt @C = 16pt{
	 & & \bullet \ar@{-}[d] & & \\
	\circ \ar@{-}[r] & \circ \ar@{-}[r] & \circ \ar@{-}[r] & \circ \ar@{-}[r] & \circ
}
$$
or
$$
(2)
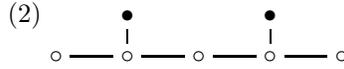
\xymatrix@R = 5pt @C = 16pt{
	 & \bullet \ar@{-}[d] & & \bullet \ar@{-}[d] & \\
	\circ \ar@{-}[r] & \circ \ar@{-}[r] & \circ \ar@{-}[r] & \circ \ar@{-}[r] & \circ
}
$$
	\item[6.] If $X$ is type $A_6$, $V$ is a RDP Del Pezzo surface of degree $4$ with $2A_1 + A_2$ singularities.
$$
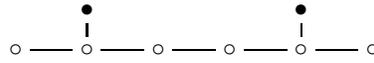
\xymatrix@R = 5pt @C = 16pt{
	 & \bullet \ar@{-}[d] & & & \bullet \ar@{-}[d] & \\
	\circ \ar@{-}[r] & \circ \ar@{-}[r] & \circ \ar@{-}[r] & \circ \ar@{-}[r] & \circ \ar@{-}[r] & \circ
}
$$
	\item[7.] If $X$ is type $A_7$, $V$ is a RDP Del Pezzo surface of degree $4$ with $2A_1 + A_3$ singularities.
$$
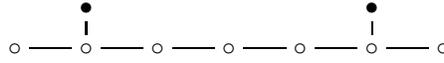
\xymatrix@R = 5pt @C = 16pt{
	 & \bullet \ar@{-}[d] & & & & \bullet \ar@{-}[d] & \\
	\circ \ar@{-}[r] & \circ \ar@{-}[r] & \circ \ar@{-}[r] & \circ \ar@{-}[r] & \circ \ar@{-}[r] & \circ \ar@{-}[r] & \circ
}
$$

	\item[8.] If $X$ is type $D_5$, $V$ is a RDP Del Pezzo surface of degree $4$ with $D_4$ singularity.
$$
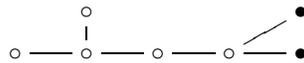
\xymatrix@R = 5pt @C = 16pt{
	 & \circ \ar@{-}[d] & & & \bullet \\
	\circ \ar@{-}[r] & \circ \ar@{-}[r] & \circ \ar@{-}[r] & \circ \ar@{-}[r] \ar@{-}[ru] & \bullet
}
$$
	\item[9.] If $X$ is type $D_6$, $V$ is a RDP Del Pezzo surface of degree $4$ with $D_5$ singularity.
$$
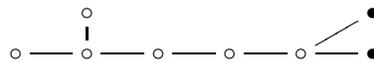
\xymatrix@R = 5pt @C = 16pt{
	 & \circ \ar@{-}[d] & & & & \bullet \\
	\circ \ar@{-}[r] & \circ \ar@{-}[r] & \circ \ar@{-}[r] & \circ \ar@{-}[r] & \circ \ar@{-}[r] \ar@{-}[ru] & \bullet
}
$$
	\item[10.] If $X$ is type $E_6$, $V$ is a RDP Del Pezzo surface of degree $4$ with $D_4$ singularity.
$$
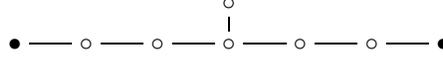
\xymatrix@R = 5pt @C = 16pt{
	 & & & \circ \ar@{-}[d] & & & \\
	\bullet \ar@{-}[r] & \circ \ar@{-}[r] & \circ \ar@{-}[r] & \circ \ar@{-}[r] & \circ \ar@{-}[r] & \circ \ar@{-}[r] & \bullet
}
$$
	\item[11.] If $X$ is type $E_7$, $V$ is a RDP Del Pezzo surface of degree $3$ with $E_6$ singularity.
$$
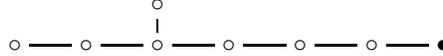
\xymatrix@R = 5pt @C = 16pt{
	 & & \circ \ar@{-}[d] & & & \\
	\circ \ar@{-}[r] & \circ \ar@{-}[r] & \circ \ar@{-}[r] & \circ \ar@{-}[r] & \circ \ar@{-}[r] & \circ \ar@{-}[r] & \bullet
}
$$
\end{itemize}
\end{thmA5}
\begin{remA6}
	In particular, the unirationality or rationality of $X$ is following:
\begin{itemize}
	\item[] in cases 2, $X$ is $k$-unirational if $\Xt$ has a $k$-point;
	\item[] in cases 3 and 4(1), $X$ is $k$-unirational;
	\item[] in cases 4(2), 6 and 7, $X$ is $k$-rational if $\Xt$ has a $k$-point;
	\item[] in cases 1, 5, 8, 9, 10 and 11, $X$ is $k$-rational.
\end{itemize}
\end{remA6}
\section*{Acknowledgement}
	I would like to thank my supervisor Keiji Oguiso for suggesting to me the theme of this article and providing many valuable comments. I would also like to thank Masaru Nagaoka who gave me very important comments and examples in {\cite{KN}}. Also, I thank Tatsuro Kawakami for letting me know Proposition {\ref{Prop:Kawakami}}.

	This research did not receive any specific grant from funding agencies in the public, commercial, or not-for-profit sectors.

\end{document}